\documentclass[10pt]{amsart}
      \usepackage[mathscr]{eucal}
      \usepackage{amsmath,amsfonts}
      \parskip\smallskipamount
          \newtheorem{theorem}{Theorem}[section]
      \newtheorem{definition}[theorem]{Definition}
      \newtheorem{proposition}[theorem]{Proposition}
      \newtheorem{corollary}[theorem]{Corollary}
      \newtheorem{lemma}[theorem]{Lemma}

      \makeatletter
      \@addtoreset{equation}{section}
      \makeatother

      \newcommand{\BB}{{\mathbb B}}
      \newcommand{\CC}{{\mathbb C}}
      \newcommand{\NN}{{\mathbb N}}
      
      \newcommand{\ZZ}{{\mathbb Z}}
      \newcommand{\DD}{{\mathbb D}}
      
      \newcommand{\FF}{{\mathbb F}}
      \newcommand{\TT}{{\mathbb T}}

      \newcommand{\cA}{{\mathcal A}}
      
      \newcommand{\cC}{{\mathcal C}}
      \newcommand{\cD}{{\mathcal D}}
      \newcommand{\cE}{{\mathcal E}}
      \newcommand{\cF}{{\mathcal F}}
      \newcommand{\cG}{{\mathcal G}}
      \newcommand{\cH}{{\mathcal H}}
      
      \newcommand{\cK}{{\mathcal K}}
       \newcommand{\cJ}{{\mathcal J}}
      
      \newcommand{\cM}{{\mathcal M}}

      \newcommand{\cP}{{\mathcal P}}
      
      \newcommand{\cS}{{\mathcal S}}
      \newcommand{\cT}{{\mathcal T}}

      \newdimen\expt
      \def\boxit#1{\setbox0\hbox{$\displaystyle{#1}$}
            \hbox{\lower.4\expt
       \hbox{\lower3\expt\hbox{\lower\dp0
            \hbox{\vbox{\hrule height.4\expt
       \hbox{\vrule width.4\expt\hskip3\expt
            \vbox{\vskip3\expt\box0\vskip2\expt}%
       \hskip3\expt\vrule width.4\expt}\hrule height.4\expt}}}}}}
      
\begin{document}
       \pagestyle{myheadings}
      \markboth{ Gelu Popescu}{Multi-Toeplitz operators on noncommutative hyperballs}

      \title [Multi-Toeplitz operators associated with regular polydomains]
      { Multi-Toeplitz operators associated with regular polydomains  }
        \author{Gelu Popescu}
\date{February 15, 2019}
     \thanks{Research supported in part by  NSF grant DMS 1500922}
      \address{Department of Mathematics, The University of Texas
      at San Antonio \\ San Antonio, TX 78249, USA}
      \email{\tt gelu.popescu@utsa.edu}

\begin{abstract}
In this paper we introduce and study the class of weighted multi-Toeplitz operators associated with noncommutative polydomains ${\bf D_f^m}$, ${\bf m}:=(m_1,\ldots, m_k)\in \NN^k$, generated by $k$-tuples  ${\bf f}:=(f_1,\ldots, f_k)$ of positive regular free holomorphic functions in a neighborhood of the origin. These operators are acting on the tensor product $F^2(H_{n_1})\otimes \cdots \otimes F^2(H_{n_k})$ of full Fock spaces with $n_i$ generators or, equivalently, they can be viewed as multi-Toeplitz operators acting  on tensor products of weighted full Fock spaces. For a large class of polydomains, we show that there are no non-zero compact multi-Toeplitz  operators.

We characterize the weighted multi-Toeplitz operators in terms of bounded free $k$-pluriharmonic functions on the radial part of ${\bf D_f^m}$ and use the result to obtain an analogue of the Dirichlet extension problem for free $k$-pluriharmonic functions.
We   show that the weighted multi-Toeplitz operators have noncommutative Fourier representations which can be viewed as noncommutative symbols and can be used to recover the associated  operators. We also prove that the  weighted multi-Toeplitz operators satisfy a Brown-Halmos type equation associated with the polydomain ${\bf D_f^m}$.

All the results hold, in particular, for the multi-Toeplitz operators acting on the  reproducing kernel Hilbert space with    reproducing kernel
 $$
 \kappa_{\bf m}(z,w):=\prod_{i=1}^k \frac{1}{\left(1-\sum_{p=1}^{N_i} a_{i,p} \bar z_i^pw_i^p\right)^{m_i}}, \qquad z=(z_1,\ldots, z_k), w=(w_1,\ldots, w_k) \in \DD^k,
 $$
 where $a_{i,p}\geq 0$ and $a_{i,1}>0$.
\end{abstract}

      \maketitle

\bigskip

\section*{Introduction}

 The study of Toeplitz operators  was initiated  by  O.Toeplitz \cite{T} and  was continued,  later on,  with the seminal work of Brown and Halmos \cite{BH}. Over the years, this study has been  extended  to Toeplitz operators  acting on  Hilbert spaces of holomorphic functions  on the unit disc (see \cite{HKZ}) such as the Bergman space and weighted Bergman space,
 and also to higher dimensional setting  involving holomorphic functions in several complex variables on various classes of domains in $\CC^n$ (see Upmeier's book \cite{U}). We refer the reader to \cite{BS}, \cite{Dou}, \cite{RR}, and \cite{HKZ}  for a   comprehensive account on Toeplitz operators and their applications.

   In the noncommutative multivariable setting, a study of  unweighted multi-Toeplitz operators on the full Fock  space  $F^2(H_n)$ with $n$ generators   was initiated in \cite{Po-multi}, \cite{Po-analytic} and has had an important impact in multivariable operator theory  and the structure of free semigroups algebras (see \cite{DP2}, \cite{DKP}, \cite{DLP}, \cite{Po-entropy}, \cite{Po-pluriharmonic}, \cite{Ken1}, \cite{Ken2}).
 Unweighted  multi-Toeplitz operators associated with noncommutative polyballs  and acting on tensor products of full Fock spaces were studied in \cite{Po-pluri}.  In \cite{Po-Toeplitz-poly-hyperball}, we obtained a Brown-Halmos \cite{BH} type characterization  of multi-Toeplitz operators associated with noncommutative poly-hyperballs which extended  the characterization of Toeplitz operators with harmonic symbols on the weighted Bergman space $A_m(\DD)$ obtained by Louhichi and Olofsson \cite{LO}, as well as Eschmeier and Langend\" orfer  \cite{EL}  extension  to  the unit ball  of $ \CC^n$.

Very recently \cite{Po-Toeplitz}, we initiated the study of weighted  multi-Toeplitz operators associated with noncommutative regular domains    ${\cD}_f^m(\cH)$, $m\in \NN$,   generated by    positive regular free holomorphic functions  $f$ in a neighborhood of the origin.
In the present paper, which can be seen as a continuation of \cite{Po-Toeplitz}, we introduce and study
the class of weighted multi-Toeplitz operators associated with noncommutative polydomains.

   To present our results, we need  some definitions.
Let $k\in \NN:=\{1,2,\ldots, \}$ and ${\bf n}:=(n_1,\ldots, n_k)\in \NN^k$.
For each $i\in \{1,\ldots, k\}$,
let $\FF_{n_i}^+$ be the unital free semigroup on $n_i$ generators
$g_{1}^i,\ldots, g_{n_i}^i$ and the identity $g_{0}^i$.  The length of $\alpha\in
\FF_{n_i}^+$ is defined by $|\alpha|:=0$ if $\alpha=g_0^i$  and
$|\alpha|:=p$ if
 $\alpha=g_{j_1}^i\cdots g_{j_p}^i$, where $j_1,\ldots, j_p\in \{1,\ldots, n_i\}$.
If $Z_{i,1},\ldots,Z_{i,n_i}$  are  noncommuting indeterminates,   we
denote $Z_{i,\alpha}:= Z_{i,j_1}\cdots Z_{i,j_p}$  and $Z_{g_0^i}:=1$.
 Let  $f_i:= \sum_{\alpha\in
\FF_{n_i}^+} a_{i,\alpha} Z_{i,\alpha}$, \ $a_{i,\alpha}\in \CC$,  be a formal power series in $n_i$ noncommuting indeterminates $Z_{i,1},\ldots,Z_{i,n_i}$. We say that $f_i$ is
a {\it
positive regular free holomorphic function} if the following conditions hold:
$a_{i,\alpha}\geq 0$ for
any $\alpha\in \FF_{n_i}^+$, \ $a_{i,g_{0}^i}=0$,
   $a_{i,g_{j}^i}>0$ for  $j=1,\ldots, n_i$, and
 $$
\limsup_{k\to\infty} \left( \sum_{|\alpha|=k}
a_{i,\alpha}^2\right)^{1/2k}<\infty.
 $$
Given $X_i:=(X_{i,1},\ldots, X_{i,n_i})\in B(\cH)^{n_i}$, define the map $\Phi_{f_i,X_i}:B(\cH)\to B(\cH)$  by setting
 $$
\Phi_{f_i,X_i}(Y):=\sum_{k=1}^\infty\sum_{\alpha\in \FF_{n_i}^+,|\alpha|=k} a_{i,\alpha} X_{i,\alpha}
YX_{i,\alpha}^*, \qquad   Y\in B(\cH),$$
 where  the convergence is in the week
operator topology and $B(\cH)$ is the algebra of all bounded linear operators on the Hilbert space $\cH$.

Let ${\bf m}:=(m_1,\ldots, m_k)$, where $m_i\in\NN$,   and let ${\bf f}:=(f_1,\ldots,f_k)$ be a $k$-tuple of positive regular free holomorphic functions. We denote by $B(\cH)^{n_1}\times_c\cdots \times_c B(\cH)^{n_k}$
   the set of all tuples  ${\bf X}=({ X}_1,\ldots, { X}_k)\in B(\cH)^{n_1}\times\cdots \times B(\cH)^{n_k}$      with the property that, for any $p,q\in \{1,\ldots, k\}$, $p\neq q$, the entries of ${ X}_p$ are commuting with the entries of ${ X}_q$.
  The {\it noncommutative polydomain} ${\bf D_f^m}$ is defined by its Hilbert space representations
${\bf D_f^m}(\cH)$ on Hilbert spaces $\cH$, where  ${\bf D_f^m}(\cH)$ is     the
  the set of all $k$-tuples
$${\bf X}:=(X_1,\ldots, X_k)\in   B(\cH)^{n_1}\times_c\cdots \times_c B(\cH)^{n_k}$$
with the property that
$$
(id-\Phi_{f_1,X_1})^{p_1}\circ\cdots \circ(id-\Phi_{f_k,X_k})^{p_k}(I)\geq 0,\qquad {\bf0}\leq {\bf p}\leq{\bf m}
$$
where ${\bf p}=(p_1,\ldots, p_k)$, with $p_i\in \NN\cup\{0\}$.
We use the convention that $(id-\Phi_{f_i,X_i})^0=id$.
In \cite{Po-domains-models}, we showed that  each  regular polydomain ${\bf D_f^m}$ has a {\it left universal model} ${\bf W}:=\{{\bf W}_{i,j}\}$ consisting on weighted left creation operators acting  on the tensor product $F^2(H_{n_1})\otimes \cdots \otimes F^2(H_{n_k})$ of full Fock spaces with $n_i$ generators (see Section 1 for  the definitions). This was used in \cite{Po-Berezin2}, \cite{Po-Berezin1} to develop an operator model theory and a theory of free holomorphic functions on these polydomains and  associated noncommutative varieties.
This study was continued recently in \cite{Po-invariant},  where we obtained a Beurling type characterization of the joint invariant subspaces under the universal model  ${\bf W}:=\{{\bf W}_{i,j}\}$ and developed an operator model theory for completely non-coisometric  elements in polydomains (resp. varieties) in terms of characteristic functions.

The goal of the present paper is to show that, as in the classical case of Toeplitz operators on the Hardy space $H^2(\DD)$ of the disc $\DD=\{z\in \CC: |z|<1\}$, there is a class of multi-Toeplitz operators associated with each noncommutative regular polydomain ${\bf D_f^m}$ which are acting on tensor products of full Fock spaces. We initiate  the study of these operators in close connection with the theory of free $k$-pluriharmonic functions on polydomains.

  In Section 1, we introduce the multi-homogeneous operators acting on tensor products of full Fock spaces  and the weighted multi-Toeplitz operators    associated with the regular  polydomains
${\bf D_{f}^m}$. The main result of this section is a characterization  of  the multi-homogeneous operators  which are also weighted multi-Toeplitz operators (see Theorem \ref{multi-homo}). In addition, we show that  any weighted multi-Toeplitz operator  $T\in B(\cK\bigotimes \otimes_{s=1}^k  F^2(H_{n_s}))$ has
a unique formal Fourier representation
$$
\varphi_T({\bf W}, {\bf W}^*):=\sum_{(\boldsymbol \alpha, \boldsymbol \beta)\in \boldsymbol \cJ}A_{   (\boldsymbol \alpha, \boldsymbol \beta)}\otimes {\bf W}_{\boldsymbol\alpha}{\bf W}_{\boldsymbol\beta}^*,
$$
 which can be viewed as  a  noncommutative symbol.
 At the end of this section, we show that,  for a large class of polydomains, the associated multi-Toeplitz operators contain no non-zero compact operators.

 In Section 2, we  characterize the weighted multi-Toeplitz operators in terms of bounded free $k$-pluriharmonic functions on the radial part of ${\bf D_f^m}$ (see Theorem \ref{main}) and  also characterize the formal series
    which are  noncommutative Fourier representations  of weighted multi-Toeplitz operators (see Theorem \ref{Fourier}).  These results are used,  in Section 3,   to prove that the   bounded free $k$-pluriharmonic  functions on the radial polydomain ${\bf D}_{{\bf f},rad}^{\bf m}$
  are precisely those  that are noncommutative Berezin transforms of the weighted multi-Toeplitz operators. In this setting, we solve the Dirichlet extension problem.

In Section 4, we prove that the weighted multi-Toeplitz operators satisfy a Brown-Halmos type equation associated with the polydomain ${\bf D_f^m}$, leaving the converse as an open problem.

\bigskip

\section{Multi-homogeneous operators and weighted multi-Toeplitz operators associated with polydomains }

In this section, we introduce the multi-homogeneous operators acting on tensor products of full Fock spaces  and the weighted multi-Toeplitz operators    associated with the polydomain
${\bf D_{f}^m}$. We characterize the multi-homogeneous operators  which are also weighted multi-Toeplitz operators and show that any weighted multi-Toeplitz operator has a Fourier representation which can be seen as a noncommutative symbol.
For a large class of polydomains, we show that the associated weighted multi-Toeplitz operators contain no non-zero compact operators.

  Let $H_{n_i}$ be
an $n_i$-dimensional complex  Hilbert space with orthonormal basis
$e_1^i,\dots,e_{n_i}^i$.
  We consider the full Fock space  of $H_{n_i}$ defined by
$$F^2(H_{n_i}):=\bigoplus_{p\geq 0} H_{n_i}^{\otimes p},$$
where $H_{n_i}^{\otimes 0}:=\CC 1$ and $H_{n_i}^{\otimes p}$ is the
(Hilbert) tensor product of $p$ copies of $H_{n_i}$. Set $e_\alpha^i :=
e^i_{j_1}\otimes \cdots \otimes e^i_{j_p}$ if
$\alpha=g^i_{j_1}\cdots g^i_{j_p}\in \FF_{n_i}^+$
 and $e^i_{g^i_0}:= 1\in \CC$.
Note that $\{e^i_\alpha:\alpha\in\FF_{n_i}^+\}$ is an orthonormal
basis of $F^2(H_{n_i})$.

Let $\Gamma:\TT^k\to B( \otimes_{s=1}^k F^2(H_{n_s}))$ be the strongly continuous unitary representation of the $k$-dimensional torus, defined by
$$
\Gamma(e^{i\theta_1},\ldots, e^{i\theta_k})f
:=\sum_{{\alpha_s\in \FF_{n_s}^+ }\atop{s\in \{1,\ldots,k\}}} e^{i\theta_1|\alpha_1|}\cdots e^{i\theta_k|\alpha_k|} a_{\alpha_1,\ldots, \alpha_k} e^1_{\alpha_1}\otimes \cdots \otimes e^k_{\alpha_k}.
$$
for any $f=\sum_{{\alpha_s\in \FF_{n_s}^+ }\atop{s\in \{1,\ldots,k\}}} a_{\alpha_1,\ldots, \alpha_k} e^1_{\alpha_1}\otimes \cdots \otimes e^k_{\alpha_k}\in \otimes_{s=1}^k F^2(H_{n_s})$.
The spectral subspaces $\{\cE_{p_1,\ldots, p_k}\}_{(p_1,\ldots, p_k)\in \ZZ^k}$ of $\Gamma$ are
the images of the orthogonal projections ${\bf P}_{p_1,\ldots, p_k}\in B(\otimes_{s=1}^k F^2(H_{n_s}))$ defined by
$$
{\bf P}_{p_1,\ldots, p_k}:= \left(\frac{1}{2\pi}\right)^k
\int_0^{2\pi}\cdots \int_0^{2\pi} e^{-ip_1\theta_1}\cdots e^{-ip_k\theta_k}\Gamma(e^{i\theta_1},\ldots, e^{i\theta_k})d\theta_1\ldots d\theta_k,
$$
where the integral is defined as a weak integral  and the integrant is a continuous function in  the strong operator topology. Note also that
 $$
\cE_{p_1,\ldots, p_k}:=\left\{ f\in \otimes_{s=1}^k F^2(H_{n_s}): \
\Gamma(e^{i\theta_1},\ldots, e^{i\theta_k})f=e^{i\theta_1 p_1}\cdots e^{i\theta_k p_k}f\right\}
$$
for $(p_1,\ldots, p_k)\in \ZZ^k$, and we have
 the orthogonal decomposition
$$
 \otimes_{s=1}^k F^2(H_{n_s})=\bigoplus_{(p_1,\ldots, p_k)\in \ZZ^k} \cE_{p_1,\ldots, p_k}.
 $$
We remark that if $p_s<0$ for some $s\in \{1,\ldots, k\}$, then   $\cE_{p_1,\ldots,p_k}=\{0\}$.
  In what follows, we  use the notation
$\Gamma (e^{i {\boldsymbol\theta}}):= \Gamma(e^{i\theta_1},\ldots, e^{i\theta_k})$. Let $\cK$ be a separable Hilbert space.

\begin{definition}  \label{mhp} If $T\in B(\cK\bigotimes \otimes_{s=1}^k  F^2(H_{n_s}))$ and ${\bf s}:=(s_1,\ldots, s_k)\in \ZZ^k$ we define the ${\bf s}$-multi-homogeneous part of $T$  to be the operator $T_{ \bf s}\in B(\cK\bigotimes \otimes_{s=1}^k  F^2(H_{n_s}))$ defined by
$$
T_{\bf s}:=
\left(\frac{1}{2\pi}\right)^k
\int_0^{2\pi}\cdots \int_0^{2\pi} e^{-is_1\theta_1}\cdots e^{-is_k\theta_k}\left(I_\cK\otimes \Gamma (e^{i {\boldsymbol\theta}})\right) T
\left(I_\cK\otimes \Gamma (e^{i {\boldsymbol\theta}})\right)^* d\theta_1\ldots d\theta_k.
$$

\end{definition}

\begin{proposition}\label{prop}  For any  ${\bf s}=(s_1,\ldots, s_k)\in \ZZ^k$,
the ${\bf s}$-multi-homogeneous part of $T\in B(\cK\bigotimes \otimes_{s=1}^k  F^2(H_{n_s}))$ has the following properties:
\begin{enumerate}
\item[(i)] $(T^*)_{\bf s}=(T_{-{\bf s}})^*$;

\item[(ii)] $T_{\bf s}\left(\cK\otimes \cE_{p_1,\ldots, p_k}\right)\subset  \cK\otimes \cE_{s_1+p_1,\ldots, s_k+p_k}$
 for any $(p_1,\ldots, p_k)\in \ZZ^k$.
\end{enumerate}
\end{proposition}
\begin{proof}  Part (i) is clear.
To prove part (ii), note  that
$$\Gamma (e^{i {\boldsymbol\theta}})^*|_{\cE_{p_1,\ldots, p_k}}
=e^{-ip_1\theta_1}\cdots e^{-ip_k\theta_k}I_{\cE_{p_1,\ldots, p_k}},\qquad (p_1,\ldots, p_k)\in \ZZ^k.
$$

Consequently, if  ${\bf s}=(s_1,\ldots, s_k)\in \ZZ^k$ and  $f\in \cK\otimes \cE_{p_1,\ldots, p_k}$, then
\begin{equation*}
\begin{split}
T_{\bf s}f&=\left(\frac{1}{2\pi}\right)^k
\int_0^{2\pi}\cdots \int_0^{2\pi} e^{-i(s_1+p_1)\theta_1}\cdots e^{-i(s_k+p_k)\theta_k}\left(I_\cK\otimes \Gamma (e^{i {\boldsymbol\theta}})\right) Tf d\theta_1\ldots d\theta_k\\
 &=(I_\cK\otimes{\bf P}_{s_1+p_1,\ldots, s_k+p_k})Tf.
\end{split}
\end{equation*}
Therefore,
$$T_{\bf s}\left(\cK\otimes \cE_{p_1,\ldots, p_k}\right)\subset  \cK\otimes \cE_{s_1+p_1,\ldots, s_k+p_k}$$
 for any $(p_1,\ldots, p_k)\in \ZZ^k$, which completes the proof.
\end{proof}

\begin{definition} An operator  $A\in B(\cK\bigotimes \otimes_{s=1}^k  F^2(H_{n_s}))$ is said to be  multi-homogeneous of degree ${\bf s}=(s_1,\ldots, s_k)\in \ZZ^k$ if
$$
A\left(\cK\otimes \cE_{p_1,\ldots, p_k}\right)\subset  \cK\otimes \cE_{s_1+p_1,\ldots, s_k+p_k}
$$
 for any $(p_1,\ldots, p_k)\in \ZZ^k$.
\end{definition}

Brown and Halmos \cite{BH} proved that a necessary and sufficient condition that an operator on the Hardy space $H^2(\TT)$ be a Toeplitz operator is that its matrix
 $[\lambda_{ij}]$ with respect to the standard basis $\chi_k(e^{i\theta})=e^{ik\theta}$, $k\in \{0,1,\ldots\}$, be a Toeplitz  matrix, i.e
 $$
 \lambda_{i+1,j+1}=\lambda_{ij},\qquad  i,j\in \{0,1,\ldots\},
 $$
 which is equivalent to   the fact that
 $\lambda_{ij}=a_{i-j}$, where $\varphi =\sum_{k\in \ZZ}a_k \chi_k$ is the Fourier expansion of the symbol $\varphi\in L^\infty(\TT)$.  In what follows, we find an  extension of their result to multi-Toeplitz operators associated with  noncommutative regular  polydomains.

First, we need some notations.  If $\omega, \gamma\in \FF_n^+$,
we say that $\omega
\geq_{r}\gamma$ if there is $\sigma\in
\FF_n^+ $ such that $\omega=\sigma \gamma$. In this
case  we set $\omega\backslash_r \gamma:=\sigma$. If $\sigma\neq g_0$ we write $\omega>_r \gamma$. We say that $\omega$ and $\gamma$ are {\it comparable} if either $\omega
\geq_{r}\gamma$ or $\gamma>_r\omega$.
Let $\boldsymbol\omega=(\omega_1,\ldots, \omega_k)$ and $\boldsymbol\gamma=(\gamma_1,\ldots, \gamma_k)$ be in ${\bf F}_{\bf n }^+:=\FF_{n_1}^+\times\cdots \times \FF_{n_k}^+$. We say that $\boldsymbol\omega$ and $\boldsymbol\gamma$ are comparable  if, for each $i\in \{1,\ldots, k\}$,  either one of the relations  $\omega_i<_r \gamma_i$, $\gamma_i<_r \omega_i$, or $\omega_i=\gamma_i$ holds. We also use the standard notation $s^+:=\max\{s,0\}$ and $s^-:=\max\{-s,0\}$ for any $s\in \ZZ$.

 We denote by $\boldsymbol\cC$ the set of all pairs $(\boldsymbol \sigma, \boldsymbol \beta)\in {\bf F}_{\bf n}^+\times {\bf F}_{\bf n}^+$ which are comparable, and  set
 $$
 \boldsymbol\cJ:= \left\{ (\boldsymbol \sigma, \boldsymbol \beta):=(\sigma_1,\ldots,\sigma_k,\beta_1,\ldots, \beta_k): \   \sigma_i,\beta_i\in \FF_{n_i}^+  \text{ with } |\sigma_i|=s_i^+,   |\beta_i|=s_i^-, s_i\in\ZZ\right\}.
 $$
We introduce the {\it simplification function}  ${\bf s}:\boldsymbol\cC\to \boldsymbol\cJ$ defined by
${\bf s}(\boldsymbol \omega, \boldsymbol\gamma):= (\boldsymbol \sigma,\boldsymbol \beta)$, where,   if  $\boldsymbol\omega=(\omega_1,\ldots, \omega_k)$ and $\boldsymbol\gamma=(\gamma_1,\ldots, \gamma_k)$, then
for any $i\in \{1,\ldots, k\}$,
$$
\sigma_i:=
\begin{cases} \omega_i\backslash_r\gamma_i, & \text{ if } \omega_i\geq_r \gamma_i,\\
g_0^i, &\text{otherwise,}
\end{cases}
  \quad
\beta_i:=\begin{cases} \gamma_i\backslash_r\omega_i, & \text{ if } \gamma_i\geq_r \omega_i,\\
g_0^i, &\text{otherwise}.
\end{cases}
$$

\begin{definition} \label{MT}  An operator $T\in B(\cK\bigotimes \otimes_{s=1}^k  F^2(H_{n_s}))$ is called  weighted (right) multi-Toeplitz operator associated with the polydomain ${\bf D_f^m}$  if there exist operators $\{A_{(\boldsymbol \sigma, \boldsymbol \beta)}\}_{(\boldsymbol \sigma, \boldsymbol \beta)\in \boldsymbol\cJ}\subset B(\cK)$ such that,
for any $ \boldsymbol\omega, \boldsymbol\gamma \in {\bf F}_{\bf n}^+$ and $x,y\in \cK$,
$$
\left<T(x\otimes e_{\boldsymbol\gamma }), y\otimes e_{\boldsymbol\omega} \right>
=\begin{cases}
\tau_{(\boldsymbol\omega,\boldsymbol\gamma)}\left<A_{{\bf s}(\boldsymbol \omega, \boldsymbol\gamma)}x,y\right>, & \text{ if  } (\boldsymbol \omega, \boldsymbol\gamma)\in \boldsymbol\cC,\\
0,&  \text{ if  } (\boldsymbol \omega, \boldsymbol\gamma)\in ({\bf F}_{\bf n }^+\times {\bf F}_{\bf n }^+)\backslash \boldsymbol\cC,

\end{cases}
$$
where the weights $ \{\tau_{(\boldsymbol\omega,\boldsymbol\gamma)}\}_{(\boldsymbol \omega, \boldsymbol\gamma)\in \boldsymbol\cC}$ are given by
$$ \tau_{(\boldsymbol\omega,\boldsymbol\gamma)}:=
 \prod_{i=1}^k \sqrt{\frac{b^{(m_i)}_{i,\min\{\omega_i, \gamma_i\}}}{b^{(m_i)}_{i,\max\{\omega_i, \gamma_i\}}}}
 $$
 and the coefficients $b_{i,\beta_i}^{(m_i)}$ are given by  relation \eqref{b-al2}.
  \end{definition}

  \begin{proposition}\label{Toep-def2}
  An operator $T\in B(\cK\bigotimes \otimes_{s=1}^k  F^2(H_{n_s}))$ is   weighted multi-Toeplitz  if and only if,
  for any $ \boldsymbol\omega, \boldsymbol\gamma \in  {\bf F}_{n}^+$,
\begin{equation*}
\left<T(x\otimes e_{\boldsymbol\gamma }), y\otimes e_{\boldsymbol\omega} \right>
=\begin{cases}
\frac{\tau_{(\boldsymbol\omega,\boldsymbol\gamma)}}{\tau_{(\boldsymbol \omega', \boldsymbol\gamma')}}\left<T(x\otimes e_{\boldsymbol\gamma'}), y\otimes e_{\boldsymbol\omega'} \right>, & \text{ if  } (\boldsymbol \omega, \boldsymbol\gamma)\in \boldsymbol\cC,\\
0,&  \text{ if  } (\boldsymbol \omega, \boldsymbol\gamma)\in ({\bf F}_{\bf n }^+\times {\bf F}_{\bf n }^+)\backslash \boldsymbol\cC,
\end{cases}
\end{equation*}
where  $(\boldsymbol \omega', \boldsymbol\gamma'):={\bf s}(\boldsymbol \omega, \boldsymbol\gamma)$  when $(\boldsymbol \omega, \boldsymbol\gamma)\in \boldsymbol\cC$.
  \end{proposition}
  \begin{proof}
  If  $T$ is a weighted multi-Toeplitz operator and
  $(\boldsymbol \omega', \boldsymbol\gamma')\in \boldsymbol\cJ$, then ${\bf s}(\boldsymbol \omega', \boldsymbol\gamma')=(\boldsymbol \omega', \boldsymbol\gamma')$ and
  $\tau_{(\boldsymbol \omega', \boldsymbol\gamma')}=\prod_{i=1}^k \frac{1}{\sqrt{b_{i,\max\{\boldsymbol \omega_i', \boldsymbol\gamma_i'\}}}}$.
 Using   Definition \ref{MT}, we deduce that
  $$
  \left<T(x\otimes e_{\boldsymbol\gamma' }), y\otimes e_{\boldsymbol\omega'} \right>
= \tau_{(\boldsymbol\omega',\boldsymbol\gamma')}\left<A_{(\boldsymbol \omega', \boldsymbol\gamma')}x,y\right>,\qquad x,y\in \cK.
$$
 If  $ (\boldsymbol\omega, \boldsymbol\gamma) \in \boldsymbol\cC$, we set  $(\boldsymbol \omega', \boldsymbol\gamma'):={\bf s}(\boldsymbol \omega, \boldsymbol\gamma)$. The relation above implies
 $$
 \left<A_{{\bf s}(\boldsymbol \omega, \boldsymbol\gamma)}x,y\right>=\frac{1}{\tau_{(\boldsymbol \omega', \boldsymbol\gamma')}}\left<T(x\otimes e_{\boldsymbol\gamma' }), y\otimes e_{\boldsymbol\omega'} \right>.
 $$
 Consequently, due to   Definition  \ref{MT}, we deduce that   the direct implication holds.
 To prove the converse, assume that  the relation in the proposition holds.
 Then, for any    $(\boldsymbol \omega', \boldsymbol\gamma')\in \boldsymbol\cJ$, we define
 $A_{(\boldsymbol \omega', \boldsymbol\gamma')}\in B(\cK)$  by setting
 \begin{equation*}
 \left< A_{(\boldsymbol \omega', \boldsymbol\gamma')}x,y\right>:=\frac{1}{\tau_{(\boldsymbol \omega', \boldsymbol\gamma')}}\left<T(x\otimes e_{\boldsymbol\gamma' }), y\otimes e_{\boldsymbol\omega'} \right>,\qquad x,y\in \cK.
 \end{equation*}
Since ${\bf s}(\boldsymbol \omega, \boldsymbol\gamma)\in \boldsymbol\cJ$  when $(\boldsymbol \omega, \boldsymbol\gamma)\in \boldsymbol\cC$, we can use  relation  above  when $(\boldsymbol \omega', \boldsymbol\gamma'):={\bf s}(\boldsymbol \omega, \boldsymbol\gamma)$ and  the relation in the proposition,
   to deduce that $T$ is a weighted multi-Toeplitz operator.
  The proof is complete.
  \end{proof}

For each $i\in \{1,\ldots, k\}$, set
   $b_{i,g_0^i}^{(m_i)} :=1$ and
\begin{equation}
\label{b-al2}
 b_{i,\alpha}^{(m_i)}:= \sum_{j=1}^{|\alpha|}
\sum_{{\gamma_1,\ldots,\gamma_j\in \FF_{n_i}^+}\atop{{\gamma_1\cdots \gamma_j=\alpha }\atop {|\gamma_1|\geq
1,\ldots, |\gamma_j|\geq 1}}} a_{i,\gamma_1}\cdots a_{i,\gamma_j}\left(\begin{matrix} j+m-1\\m-1\end{matrix}\right)
   \qquad
\text{ if } \ \alpha\in \FF_{n_i}^+,  |\alpha|\geq 1.
\end{equation}
The diagonal operators  $D_{i,j}:F^2(H_{n_i})\to F^2(H_{n_i})$, $j\in \{1,\ldots, n_i\}$,   are defined  by setting
$$
D_{i,j}e^i_\alpha:=\sqrt{\frac{b_{i,\alpha}^{(m_i)}}{b_{i, g_j^i\alpha }^{(m_i)}}}
e^i_\alpha,\qquad
 \alpha\in \FF_{n_i}^+,
$$
where
$\{e^i_\alpha\}_{\alpha\in \FF_{n_i}^+}$ is the orthonormal basis of the full Fock space $F^2(H_{n_i})$.
    As in \cite{Po-domains-models}, we associate with the noncommutative domain
$$
\cD_{n_i}^{m_i}(\cH):=\{ X_i\in B(\cH)^{n_i}: \  (id-\Phi_{f_iX_i})^{p}(I)\geq 0 \text{ for }  1\leq p\leq m_i\}.
$$
The {\it weighted left creation  operators} $W_{i,j}:F^2(H_{n_i})\to
F^2(H_{n_i})$  are defined
by $W_{i,j}:=S_{i,j}D_{ij}$, where
 $S_{i,1},\ldots, S_{i,n_i}$ are the left creation operators on the full
 Fock space $F^2(H_{n_i})$, i.e.
      $$
       S_{i,j} \varphi:=e_j^i \otimes \varphi, \qquad  \varphi\in F^2(H_{n_i}).
      $$
      If $\beta=g_{j_1}^i\cdots g_{j_p}^i\in\FF_{n_i}^+$, we set $W_{i,\beta}:=W_{i,j_1}\cdots W_{i,j_p}$, and  $W_{i,g_0^i}:=I$.
       A simple calculation reveals that
\begin{equation*}
W_{i,\beta} e^i_\alpha= \frac {\sqrt{b_{i,\alpha}^{(m_i)}}}{\sqrt{b_{i,\beta
\alpha}^{(m_i)}}} e^i_{\beta \alpha} \quad \text{ and }\quad W_{i,\beta}^*
e^i_\alpha =\begin{cases} \frac
{\sqrt{b_{i,\gamma}^{(m_i)}}}{\sqrt{b_{i,\alpha}^{(m_i)}}}e^i_\gamma& \text{ if
}
\alpha=\beta\gamma \\
0& \text{ otherwise }
\end{cases}
\end{equation*}
 for any $\alpha, \beta \in \FF_{n_i}^+$.
  For each $i\in \{1,\ldots, k\}$ and $j\in \{1,\ldots, n_i\}$, we
define the operator ${\bf W}_{i,j}$ acting on the tensor Hilbert space
$F^2(H_{n_1})\otimes\cdots\otimes F^2(H_{n_k})$ by setting
$${\bf W}_{i,j}:=\underbrace{I\otimes\cdots\otimes I}_{\text{${i-1}$
times}}\otimes W_{i,j}\otimes \underbrace{I\otimes\cdots\otimes
I}_{\text{${k-i}$ times}}.
$$
 According to \cite{Po-Berezin1}, if  ${\bf W}_i:=({\bf W}_{i,1},\ldots,{\bf W}_{i,n_i})$, then
 $$
 (id-\Phi_{f_1,{\bf W}_1})^{m_1}\circ\cdots \circ (id-\Phi_{f_k,{\bf W}_k})^{m_k}(I)={\bf P}_\CC,
 $$
  where ${\bf P}_\CC$ is the
 orthogonal projection from $\otimes_{i=1}^k F^2(H_{n_i})$ onto $\CC 1\subset \otimes_{i=1}^k F^2(H_{n_i})$, where $\CC 1$ is identified with $\CC 1\otimes\cdots \otimes \CC 1$.
     Moreover,  ${\bf W}:=({\bf W}_1,\ldots, {\bf W}_k)$ is  a pure $k$-tuple, i.e. $\Phi_{f_iW_i}^p(I)\to 0$ strongly as $p\to \infty$,  in the
noncommutative polydomain $ {\bf D_f^m}(\otimes_{i=1}^kF^2(H_{n_i}))$.
The $k$-tuple   ${\bf W}:=({\bf W}_1,\ldots, {\bf W}_k)$
 plays the role of the {\it left universal model for the  noncommutative polydomain}
${\bf D_f^m}$.
More on  noncommutative polydomains, universal models, noncommutative Berezin transforms  and their applications can be found in \cite{Po-poisson}, \cite{Po-domains-models},  \cite{Po-domains}, \cite{Po-Berezin2}, and  \cite{Po-Berezin1}.

Here is an alternative description of the multi-Toeplitz operators associated with the polydomain ${\bf D_f^m}$.
For each $i\in \{1,\ldots, k\}$, let $F_{n_i,m_i}^2$ be the Hilbert space of formal power series in noncommutative indeterminates $Z_{i,1},\ldots, Z_{i,n_i}$ with complete orthogonal basis $\{Z_{i,\alpha}: \ \alpha \in \FF_{n_i}^+\}$  with the property  that $\|Z_{i,\alpha}\|_{i,m_i}:=\frac{1}{\sqrt{b_{i,\alpha}^{(m_i)}}}$.  It is clear that
$$
F_{n_i,m_i}^2=\left\{ \varphi:=\sum_{\alpha\in \FF_{n_i}^+} a_\alpha Z_{i,\alpha}: \ a_\alpha\in \CC \ \text{\rm and }\  \|\varphi\|_{i,m_i}^2:=
\sum_{\alpha\in \FF_{n_i}^+} \frac{1}{b_{i,\alpha}^{(m_i)}} |a_\alpha|^2<\infty\right\}.
$$
Note that $F_{n_i,m_i}^2$
  can be seen as a weighted Fock space with $n_i$ generators.
The left multiplication operators $L_{i,1} ,\ldots, L_{i,n_i} $ are defined by
$L_{i ,j}\xi:=Z_{i\,j}\xi$ \, for all $\xi\in F^2_{i,m_i}$.
 For each $i\in \{1,\ldots, k\}$ and $j\in \{1,\ldots, n_i\}$, we
define the operator ${\bf L}_{i,j}$ acting on the tensor Hilbert space
$F^2_{n_1,m_1}\otimes\cdots\otimes F^2_{n_k,m_k}$ by setting
$${\bf L}_{i,j}:=\underbrace{I\otimes\cdots\otimes I}_{\text{${i-1}$
times}}\otimes L_{i,j}\otimes \underbrace{I\otimes\cdots\otimes
I}_{\text{${k-i}$ times}}.
$$

 We remark  that
$U_{i,m_i}:F^2(H_{n_i})\to F^2_{n_i,m_i}$  defined by
$
U_{i,m_i}(e^i_\alpha):=\sqrt{b_\alpha^{(m_i)}} Z_{i,\alpha}$, $ \alpha\in \FF_{n_i}^+,
$
is  a unitary   operator and
   $U_{i,m_i}W_{i,j}=L_{i,j} U_{i,m_i}$  for any $ j\in \{1,\ldots, n_i\}.
$
Now, it is easy to see that  the operator  ${\bf U}:=U_{1,m_1}\otimes\cdots \otimes U_{k,m_k}:\otimes_{i=1}^k F^2(H_{n_i})\to\otimes_{i=1}^k F^2_{n_i,m_i}$
is unitary and
   ${\bf U}{\bf W}_{i,j}={\bf L}_{i,j} {\bf U}$  for any  $i\in \{1,\ldots, k\}$ and $ j\in \{1,\ldots, n_i\}.
$
A straightforward calculation shows  that $T\in B(\cK\bigotimes  \otimes_{i=1}^k F^2(H_{n_i}))$  is  a  weighted multi-Toeplitz operator if and only if  there exist operators $\{A_{(\boldsymbol \sigma, \boldsymbol \beta)}\}_{(\boldsymbol \sigma, \boldsymbol \beta)\in \boldsymbol\cJ}\subset B(\cK)$ such that the operator $T':=(I\otimes {\bf U})T(I\otimes {\bf U}^*)$   satisfies the relation
$$
\left<T'(x\otimes {\bf Z}_{\boldsymbol\gamma }), y\otimes {\bf Z}_{\boldsymbol\omega} \right>
=\begin{cases}
\mu_{(\boldsymbol\omega,\boldsymbol\gamma)}\left<A_{{\bf s}(\boldsymbol \omega, \boldsymbol\gamma)}x,y\right>, & \text{ if  } (\boldsymbol \omega, \boldsymbol\gamma)\in \boldsymbol\cC,\\
0,&  \text{ if  } (\boldsymbol \omega, \boldsymbol\gamma)\in ({\bf F}_{\bf n}^+\times {\bf F}_{\bf n}^+)\backslash \boldsymbol\cC,

\end{cases}
$$
for any $ \boldsymbol\omega, \boldsymbol\gamma \in {\bf F}_{\bf n}^+$,
where the weights $ \{\mu_{(\boldsymbol\omega,\boldsymbol\gamma)}\}_{(\boldsymbol \omega, \boldsymbol\gamma)\in \boldsymbol\cC}$ are given by
$$ \mu_{(\boldsymbol\omega,\boldsymbol\gamma)}:=
 \prod_{i=1}^k  \frac{1}{b^{(m_i)}_{i,\max\{\omega_i, \gamma_i\}}}.
 $$
We remark that all the results of  the present paper can be written in the setting of multi-Toeplitz operators on tensor products of weighted Fock spaces. Let us  mention a few particular cases which show that our definition of multi-Toeplitz operators seems to be natural for our setting.

 In the particular case when $k=1$ and $n_1=1$, the space $F^2_{1,m_1}$  coincides with the weighted Bergman space $A_{m_1}(\DD)$. As seen in \cite{Po-Toeplitz-poly-hyperball},  $T'$ is a  Toeplitz operator with operator-valued bounded harmonic symbol on $\DD$.     In the scalar case when $\cK=\CC$,  this class of operators was studied  by Louhichi and Olofsson   in  \cite{LO}.

 When $n_i=m_i=1$ for $i\in\{1,\ldots, k\}$, the tensor product  $F^2_{1,1}\otimes \cdots \otimes F^2_{1,1}$ is identified with the Hardy space $H^2(\DD^k)$.
 In this case, $T'$ is a multi-Toeplitz operator if and only if $T'=P_{H^2(\DD^k)} M_\varphi |_{H^2(\DD^k)}$ for some
 $\varphi\in L^\infty(\TT^k)$.  We should mention  that  a Brown-Halmos type characterization of Toeplitz operators on $H^2(\DD^k)$ was  recently obtained in \cite{MSS}.

 In the particular  case when $n_i=1$  for $i\in \{1,\ldots, k\}$, ${\bf m}=(m_1,\ldots, m_k)\in \NN^k$,  and  $f_i=\sum_{p=1}^{N_i} a_{i,p} z_i^p$ is a  regular polynomial in $z_i$, the tensor product $F^2_{1,m_1}\otimes \cdots \otimes F^2_{1,m_k}$ is identified with the  reproducing kernel Hilbert space with    reproducing kernel
 $$
 \kappa_{\bf m}(z,w):=\prod_{i=1}^k \frac{1}{\left(1-\sum_{p=1}^{N_i} a_{i,p} \bar z_i^pw_i^p\right)^{m_i}}, \qquad z=(z_1,\ldots, z_k), w=(w_1,\ldots, w_k) \in \DD^k.
 $$
  All the   results  of the present paper  hold,  in particular,  for these reproducing kernel Hilbert spaces, which
 include the  Hardy space, the Bergman space, and the weighted Bergman space  over the polydisk.

In what follows,  we set
${\bf W}_{\boldsymbol \alpha}:= {\bf W}_{1,\alpha_1}\cdots {\bf W}_{k,\alpha_k}$, if $\boldsymbol\alpha=(\alpha_1,\ldots, \alpha_2)\in {\bf F}_{\bf n}^+:=\FF_{n_1}^+\times\cdots\times \FF_{n_k}^+$. We recall from \cite{Po-Toeplitz-poly-hyperball}  the following straightforward technical result which follows right away using the definition of the universal model ${\bf W}$.
   \begin{lemma}
 \label{monom1}
  If  $\boldsymbol\gamma \in {\bf F}_{\bf n }^+$,   then the family $\{ {\bf W}_{\boldsymbol \alpha}{\bf W}_{\boldsymbol \beta}^*e_{\boldsymbol \gamma}\}_{(\boldsymbol \alpha, \boldsymbol \beta)\in \boldsymbol\cJ}$ consists of pairwise orthogonal vectors.
  Moreover, if  $ (\boldsymbol \alpha,\boldsymbol \beta)\in \boldsymbol\cJ$ and $\boldsymbol\omega, \boldsymbol\gamma \in {\bf F}_{\bf n }^+$, then
  $$
  \left<{\bf W}_{\boldsymbol \alpha}{\bf W}_{\boldsymbol \beta}^*e_{\boldsymbol \gamma}, e_{\boldsymbol\omega}\right>\neq 0
  $$
  if and only if  $(\boldsymbol \omega, \boldsymbol\gamma)\in \boldsymbol\cC$ and
  ${\bf s}(\boldsymbol \omega, \boldsymbol\gamma)=(\boldsymbol\alpha,\boldsymbol\beta)$.
  \end{lemma}

If  ${\bf s}:=(s_1,\ldots, s_k)\in \ZZ^k$, we use the notation
$$
 \boldsymbol\cJ_{\bf s}:= \left\{ (\boldsymbol \alpha, \boldsymbol \beta):=(\alpha_1,\ldots,\alpha_k,\beta_1,\ldots, \beta_k): \   \alpha_i,\beta_i\in \FF_{n_i}^+  \text{ with } |\alpha_i|=s_i^+,   |\beta_i|=s_i^-\right\}.
 $$
Note that $\boldsymbol\cJ=\cup_{{\bf s}\in \ZZ^k} \boldsymbol\cJ_{\bf s}$.
 We remark that the  operator
$$
 q_{\bf s }({\bf W}, {\bf W}^*):=
\sum_{ (\boldsymbol \alpha, \boldsymbol \beta)\in \boldsymbol\cJ_{\bf s}} A_{  (\boldsymbol \alpha, \boldsymbol \beta) }\otimes {\bf W}_{\boldsymbol \alpha}{\bf W}_{\boldsymbol \beta}^*,
$$
where  $A_{ (\boldsymbol \alpha, \boldsymbol \beta)}\in B(\cK)$   is a weighted  multi-Toeplitz operator and also a multi-homogeneous  operator of degree ${\bf s}$. This fact can be checked easily using the definition of the universal model ${\bf W}$. The next result shows that any weighted multi-Toeplitz operator
which is multi-homogeneous  operator of degree ${\bf s}\in \ZZ^k$ has this form.

\begin{theorem} \label{multi-homo}
 Let ${\bf s}:=(s_1,\ldots, s_k)\in \ZZ^k$ and let
  $A\in B(\cK\bigotimes \otimes_{s=1}^k  F^2(H_{n_s}))$ be a  multi-homogeneous  operator of degree  ${\bf s}$.
If $A$ is a weighted multi-Toeplitz operator, then
 $$
 Af=q_{\bf s}({\bf W}, {\bf W}^*) f,\qquad f\in \cK\bigotimes \otimes_{s=1}^k  F^2(H_{n_s}),
 $$
 where
 $$
 q_{\bf s}({\bf W}, {\bf W}^*):=
\sum_{ (\boldsymbol \alpha, \boldsymbol \beta)\in \boldsymbol\cJ_{\bf s}} C_{ (\boldsymbol \alpha, \boldsymbol \beta)}\otimes {\bf W}_{\boldsymbol \alpha}{\bf W}_{\boldsymbol \beta}^*,
$$
and  the coefficients $C_{ (\boldsymbol \alpha, \boldsymbol \beta)}\in B(\cK)$ are given by
\begin{equation*}
\left<C_{(\alpha_1,\ldots,\alpha_k,\beta_1,\ldots, \beta_k)}h,\ell\right>:=\left(\prod_{i=1}^k \sqrt{b_{i,\alpha_i}^{(m_i)} b_{i,\beta_i}^{(m_i)}}\right)\left<A(h\otimes x), \ell\otimes y\right>,\qquad  h,\ell\in \cK,
\end{equation*}
where $x:=x_1\otimes \cdots \otimes x_k$, $y=y_1\otimes \cdots \otimes y_k$ with
\begin{equation*}
\begin{cases} x_i=e^i_{\beta_i} \text{ and } y_i=1,& \quad \text{if } s_i\leq 0\\
 x_i=1 \text{ and } y_i=e^i_{\alpha_i},& \quad \text{if } s_i>0
 \end{cases}
 \end{equation*}
 for every $i\in \{1,\ldots, k\}$.
\end{theorem}
\begin{proof} First, we prove that  any weighted multi-Toeplitz operator  $T\in B(\cK\bigotimes \otimes_{s=1}^k  F^2(H_{n_s}))$ has
a unique formal Fourier representation
$$
\varphi_T({\bf W}, {\bf W}^*):=\sum_{(\boldsymbol \alpha, \boldsymbol \beta)\in \boldsymbol \cJ}A_{   (\boldsymbol \alpha, \boldsymbol \beta)}\otimes {\bf W}_{\boldsymbol\alpha}{\bf W}_{\boldsymbol\beta}^*,
$$
where $\{A_{   (\boldsymbol \alpha, \boldsymbol \beta)}\}_{(\boldsymbol \alpha, \boldsymbol \beta)\in \boldsymbol \cJ}$ are some operators on the Hilbert space $\cK$,  such that
$$
T\zeta=\varphi_T({\bf W}, {\bf W}^*)\zeta,\qquad \zeta\in \cP_K,
$$
where $\cP_\cK$ is the linear span of all vectors of the form
$h\otimes e^1_{\alpha_1}\otimes\cdots\otimes e_{\alpha_k}^k$, where $h\in \cK$, $\alpha_i\in \FF_{n_i}^+$.
Due to Definition \ref{MT},  there exist operators $\{A_{(\boldsymbol \alpha, \boldsymbol \beta)}\}_{(\boldsymbol \alpha, \boldsymbol \beta)\in \boldsymbol \cJ}\subset B(\cK)$,    such that
\begin{equation*}
 \left<A_{(\boldsymbol \alpha, \boldsymbol \beta)}h,\ell\right>=\frac{1}{\tau_{(\boldsymbol \alpha, \boldsymbol \beta)}}\left<T(h\otimes e_{\boldsymbol\beta}), \ell\otimes e_{\boldsymbol\alpha} \right>,\qquad h,\ell\in \cK.
 \end{equation*}
 A closer look at the later relation, reveals that the coefficients
 $A_{(\boldsymbol \alpha, \boldsymbol \beta)}$ are satisfying the relation
 \begin{equation*}
\left<A_{(\alpha_1,\ldots,\alpha_k,\beta_1,\ldots, \beta_k)}h,\ell\right>=\left(\prod_{i=1}^k \sqrt{b_{i,\alpha_i}^{(m_i)} b_{i,\beta_i}^{(m_i)}}\right)\left<T(h\otimes x), \ell\otimes y\right>,\qquad  h,\ell\in \cK,
\end{equation*}
where $x:=x_1\otimes \cdots \otimes x_k$, $y=y_1\otimes \cdots \otimes y_k$ with
\begin{equation*}
\begin{cases} x_i=e^i_{\beta_i} \text{ and } y_i=1,& \quad \text{if } s_i\leq 0\\
 x_i=1 \text{ and } y_i=e^i_{\alpha_i},& \quad \text{if } s_i>0
 \end{cases}
 \end{equation*}
 for every $i\in \{1,\ldots, k\}$.
Consequently, if $\boldsymbol\gamma\in {\bf F}_{\bf n }^+$ and $h\in \cK$, we have
$$
T(h\otimes e_{\boldsymbol\gamma }) =\sum_{\boldsymbol\omega\in {\bf F}_{\bf n }^+: (\boldsymbol\omega, \boldsymbol\gamma)\in \boldsymbol \cC}
 \tau_{(\boldsymbol\omega,\boldsymbol\gamma)}A_{{\bf s}(\boldsymbol \omega, \boldsymbol\gamma)}h \otimes e_{\boldsymbol\omega }
$$
is a vector in $\cK\bigotimes \otimes_{i=1}^k F^2(H_{n_i})$.  This shows that the series
\begin{equation}
\label{con}
\sum_{\boldsymbol\omega\in {\bf F}_{\bf n }^+: (\boldsymbol\omega, \boldsymbol\gamma)\in\boldsymbol  \cC}
 \tau_{(\boldsymbol\omega,\boldsymbol\gamma)}^2A_{{\bf s}(\boldsymbol \omega, \boldsymbol\gamma)}^*A_{{\bf s}(\boldsymbol \omega, \boldsymbol\gamma)}\quad \text{ is WOT-convergent}.
\end{equation}

Consider the formal power series
$$
\varphi_T({\bf W}, {\bf W}^*):=\sum_{(\boldsymbol \alpha, \boldsymbol \beta)\in \boldsymbol \cJ}A_{   (\boldsymbol \alpha, \boldsymbol \beta)}\otimes {\bf W}_{\boldsymbol\alpha}{\bf W}_{\boldsymbol\beta}^*.
$$
In what follows, we show that
$$
\varphi_T({\bf W}, {\bf W}^*)(x\otimes e_{\boldsymbol \gamma}):=\sum_{(\boldsymbol \alpha, \boldsymbol \beta)\in \boldsymbol \cJ}A_{   (\boldsymbol \alpha, \boldsymbol \beta)}x\otimes {\bf W}_{\boldsymbol\alpha}{\bf W}_{\boldsymbol\beta}^*e_{\boldsymbol \gamma}
$$
is convergent for any $x\in \cK$ and $\boldsymbol\gamma\in {\bf F}_{\bf n }^+$. According  to Lemma \ref{monom1},
if  $\boldsymbol\omega, \boldsymbol\gamma \in {\bf F}_{\bf n }^+$, then
  $$
  \left<{\bf W}_{\boldsymbol\alpha}{\bf W}_{\boldsymbol\beta}^*e_{\boldsymbol \gamma}, e_{\boldsymbol\omega}\right>\neq 0
  $$
  if and only if  $(\boldsymbol \omega, \boldsymbol\gamma)\in \boldsymbol \cC$ and
  ${\bf s}(\boldsymbol \omega, \boldsymbol\gamma)=(\boldsymbol\alpha,\boldsymbol\beta)$. In this case, a straightforward  computation reveals that
  $\left<{\bf W}_{\boldsymbol\alpha}{\bf W}_{\boldsymbol\beta}^*e_{\boldsymbol \gamma}, e_{\boldsymbol\omega}\right>=\tau_{(\boldsymbol\omega, \boldsymbol\gamma)}$.
Due to   Parseval's identity, we  have
\begin{equation*}
\begin{split}
\|{\bf W}_{\boldsymbol\alpha}{\bf W}_{\boldsymbol\beta}^*e_{\boldsymbol \gamma}\|^2
&=\sum_{\omega\in {\bf F}_{\bf n }^+} \left|\left<{\bf W}_{\boldsymbol\alpha}{\bf W}_{\boldsymbol\beta}^*e_{\boldsymbol \gamma},e_{\boldsymbol\omega}\right>\right|^2 \\
&=
\sum_{{\omega\in {\bf F}_{\bf n }^+: (\boldsymbol \omega, \boldsymbol\gamma)\in \boldsymbol \cC}\atop{
  {\bf s}(\boldsymbol \omega, \boldsymbol\gamma)=(\boldsymbol\alpha,\boldsymbol\beta)}} \left|\left<{\bf W}_{\boldsymbol\alpha}{\bf W}_{\boldsymbol\beta}^*e_{\boldsymbol \gamma},e_{\boldsymbol\omega}\right>\right|^2
  =\sum_{{\omega\in {\bf F}_{\bf n }^+: (\boldsymbol \omega, \boldsymbol\gamma)\in \boldsymbol \cC}\atop{
  {\bf s}(\boldsymbol \omega, \boldsymbol\gamma)=(\boldsymbol\alpha,\boldsymbol\beta)}} \tau_{(\boldsymbol\omega, \boldsymbol\gamma)}^2.
\end{split}
\end{equation*}
Now, using again Lemma \ref{monom1}, we deduce that
\begin{equation*}
\begin{split}
\left\|\varphi_T({\bf W}, {\bf W}^*)(x\otimes e_{\boldsymbol \gamma})\right\|^2
&=
\sum_{{(\boldsymbol\alpha,\boldsymbol\beta)\in\boldsymbol  \cJ }}
\|A_{   (\boldsymbol \alpha, \boldsymbol \beta)}x\|^2\sum_{{\omega\in {\bf F}_{\bf n }^+: (\boldsymbol \omega, \boldsymbol\gamma)\in \boldsymbol \cC}\atop{
  {\bf s}(\boldsymbol \omega, \boldsymbol\gamma)=(\boldsymbol\alpha,\boldsymbol\beta)}} \tau_{(\boldsymbol\omega, \boldsymbol\gamma)}^2\\
  &=\sum_{{\omega\in {\bf F}_{\bf n }^+: (\boldsymbol \omega, \boldsymbol\gamma)\in\boldsymbol  \cC}}
  \|A_{{\bf s}(\boldsymbol \omega, \boldsymbol\gamma)}x\|^2\tau_{(\boldsymbol\omega, \boldsymbol\gamma)}^2.
  \end{split}
\end{equation*}
Note that the later expression is finite due to relation \eqref{con}.
Using  Definition  \ref{MT}  and the results above, we deduce that
\begin{equation*}
\begin{split}
\left<\varphi_T({\bf W}, {\bf W}^*)(x\otimes e_{\boldsymbol \gamma}, y\otimes e_{\boldsymbol\omega}
\right>
&=
\sum_{(\boldsymbol \alpha, \boldsymbol \beta)\in \boldsymbol \cJ}\left<A_{   (\boldsymbol \alpha, \boldsymbol \beta)}x,y\right>\left< {\bf W}_{\boldsymbol\alpha}{\bf W}_{\boldsymbol\beta}^*e_{\boldsymbol \gamma}, e_{\boldsymbol\omega}\right>\\
&=\begin{cases}
\tau_{(\boldsymbol\omega,\boldsymbol\gamma)}\left<A_{{\bf s}(\boldsymbol \omega, \boldsymbol\gamma)}x,y\right>, & \text{ if  } (\boldsymbol \omega, \boldsymbol\gamma)\in \boldsymbol \cC,\\
0,&  \text{ if  } (\boldsymbol \omega, \boldsymbol\gamma)\in ({\bf F}_{\bf n }^+\times {\bf F}_{\bf n }^+)\backslash \boldsymbol \cC,
\end{cases}
\\
&=
\left<T(x\otimes e_{\boldsymbol\gamma }), y\otimes e_{\boldsymbol\omega} \right>.
\end{split}
\end{equation*}
Consequently, we have  $\varphi_T({\bf W}, {\bf W}^*)(x\otimes e_{\boldsymbol \gamma})=T(x\otimes e_{\boldsymbol\gamma })$ for any $x\in \cK$ and $\boldsymbol\gamma \in {\bf F}_{\bf n }^+$.
It remains to prove  the uniqueness of the Fourier representation. To this end,  assume that
$
\varphi({\bf W}, {\bf W}^*):=\sum_{(\boldsymbol \alpha, \boldsymbol \beta)\in \boldsymbol \cJ}A_{   (\boldsymbol \alpha, \boldsymbol \beta)}'\otimes {\bf W}_{\boldsymbol\alpha}{\bf W}_{\boldsymbol\beta}^*
$
is a formal series such that $T\zeta=\varphi({\bf W}, {\bf W}^*)\zeta$ for any $\zeta\in \cK$.   On the other hand,  if $(\boldsymbol \alpha, \boldsymbol \beta)\in \boldsymbol \cJ$ then ${\bf s}(\boldsymbol \alpha, \boldsymbol \beta)=(\boldsymbol \alpha, \boldsymbol \beta)$. In this case, we have
$$
\left<\varphi_T({\bf W}, {\bf W}^*)(x\otimes e_{\boldsymbol \beta}, y\otimes e_{\boldsymbol\alpha}
\right>=\tau_{(\boldsymbol \alpha, \boldsymbol \beta))}\left<A_{(\boldsymbol \alpha, \boldsymbol \beta)}x,y\right>
$$
and
$$
\left<\varphi({\bf W}, {\bf W}^*)(x\otimes e_{\boldsymbol \beta}, y\otimes e_{\boldsymbol\alpha}
\right>=\tau_{(\boldsymbol \alpha, \boldsymbol \beta)}\left<A_{(\boldsymbol \alpha, \boldsymbol \beta)}'x,y\right>.
$$
Since $\varphi_T({\bf W}, {\bf W}^*)\zeta=\varphi({\bf W}, {\bf W}^*)\zeta$ and $\tau_{(\boldsymbol \alpha, \boldsymbol \beta)}\neq 0$, the relations above imply
$A_{(\boldsymbol \alpha, \boldsymbol \beta)}=A_{(\boldsymbol \alpha, \boldsymbol \beta)}'$  for any
 $(\boldsymbol \alpha, \boldsymbol \beta)\in \boldsymbol \cJ$, which proves our assertion.

 Now, assume that $A\in B(\cK\bigotimes \otimes_{s=1}^k  F^2(H_{n_s}))$  is a weighted  multi-Toeplitz operator which is also  multi-homogeneous  operator of degree  ${\bf s}=(s_1,\ldots, s_k)\in \ZZ^k$.
 Due to the first part of the proof, $A$ has a unique  formal Fourier representation
 $$
\varphi_A({\bf W}, {\bf W}^*):= \sum_{(\boldsymbol \alpha, \boldsymbol \beta)\in \boldsymbol \cJ}C_{   (\boldsymbol \alpha, \boldsymbol \beta)}\otimes {\bf W}_{\boldsymbol\alpha}{\bf W}_{\boldsymbol\beta}^*=\sum_{{\bf t}=(t_1,\ldots, t_k)\in \ZZ^k} q_{\bf t}({\bf W}, {\bf W}^*),
$$
where
 $$
 q_{\bf t}({\bf W}, {\bf W}^*):=
\sum_{(\boldsymbol \alpha, \boldsymbol \beta)\in \boldsymbol \cJ_{\bf t}} C_{(\boldsymbol \alpha, \boldsymbol \beta)\in \boldsymbol \cJ}\otimes {\bf W}_{\boldsymbol \alpha}{\bf W}_{\boldsymbol \beta}^*
$$
and the   coefficients $C_{(\boldsymbol \alpha, \boldsymbol \beta) }\in B(\cK)$ are defined in the theorem, and  such that
$$
A\zeta= \varphi_A({\bf W}, {\bf W}^*)\zeta, \qquad \zeta\in \cP_K.
$$
 If $\zeta\in \cK\otimes \cE_{p_1,\ldots, p_k}$, then  $q_{\bf t}({\bf W}, {\bf W}^*)\zeta\in \cK\otimes \cE_{t_1+p_1,\ldots, t_k+p_k}$ and, since  $A$ is a  multi-homogeneous  operator of degree  ${\bf s}$,  we have $A\zeta\in \cK\otimes \cE_{s_1+p_1,\ldots, s_k+p_k}$.  Consequently, since the subspaces
 $\{\cE_{p_1,\ldots, p_k}\}_{(p_1,\ldots, p_k)\in \ZZ^k}$ are orthogonal and
 $$
 A\zeta=\sum_{{\bf t}=(t_1,\ldots, t_k)\in \ZZ^k} q_{\bf t}({\bf W}, {\bf W}^*) \zeta,
$$
 we deduce that  $A\zeta=q_{\bf s}({\bf W}, {\bf W}^*) \zeta$ for any $\zeta\in  \cK\otimes \cE_{p_1,\ldots, p_k}$.
 Using the fact that
 $$
 \otimes_{s=1}^k F^2(H_{n_s})=\bigoplus_{(p_1,\ldots, p_k)\in \ZZ^k} \cE_{p_1,\ldots, p_k},
 $$
 we conclude that $A=q_{\bf s}({\bf W}, {\bf W}^*)$, which completes the proof.
 \end{proof}

We should record the following result that was proved in the proof of Theorem \ref{multi-homo}.

\begin{corollary}  \label{formal2} Any weighted multi-Toeplitz operator  $T\in B(\cK\bigotimes \otimes_{s=1}^k  F^2(H_{n_s}))$ has
a unique formal Fourier representation
$$
\varphi_T({\bf W}, {\bf W}^*):=\sum_{(\boldsymbol \alpha, \boldsymbol \beta)\in \boldsymbol \cJ}A_{   (\boldsymbol \alpha, \boldsymbol \beta)}\otimes {\bf W}_{\boldsymbol\alpha}{\bf W}_{\boldsymbol\beta}^*,
$$
where $\{A_{   (\boldsymbol \alpha, \boldsymbol \beta)}\}_{(\boldsymbol \alpha, \boldsymbol \beta)\in \boldsymbol \cJ}$ are some operators on the Hilbert space $\cK$,  such that
$$
T\zeta=\varphi_T({\bf W}, {\bf W}^*)\zeta,\qquad \zeta\in \cP_K.
$$
\end{corollary}

 The formal Fourier  series $\varphi_T({\bf W}, {\bf W}^*)$  associated with the weighted multi-Toeplitz operator $T$ can be viewed as  its  noncommutative symbol. In the next section, we provide necessary and sufficient conditions on a formal  series
 $\sum_{(\boldsymbol \alpha, \boldsymbol \beta)\in \boldsymbol \cJ}A_{   (\boldsymbol \alpha, \boldsymbol \beta)}\otimes {\bf W}_{\boldsymbol\alpha}{\bf W}_{\boldsymbol\beta}^*$  to be the Fourier series of a weighted multi-Toeplitz operator.

The last result of this section shows that for a large class of polydomains the associated weighted multi-Toeplitz operators contain no non-zero compact operators.

\begin{theorem} \label{compact}  Let    ${\bf D}_{\bf f}^{\bf m}$ be a noncommutative polydomain where the coefficients $b_{i,\alpha}^{(m_i)}$ associated to ${\bf f}$ satisfy the condition
$$
\sup_{\alpha\in \FF_{n_i}^+} \frac{b^{(m_i)}_{i,g_j^i\alpha }}{b^{(m_i)}_{i,\alpha}}<\infty,\qquad i\in \{1,\ldots, k\}, j\in \{1,\ldots, n_i\}.
$$
If $T\in B(\otimes_{s=1}^k  F^2(H_{n_s}))$ is a compact weighted multi-Toeplitz operator, then $T=0$.
\end{theorem}
\begin{proof} Let  $(\boldsymbol \omega', \boldsymbol\gamma')\in \cJ$ and let
$(\boldsymbol \omega, \boldsymbol\gamma)\in \cC$ be such that
${\bf s}(\boldsymbol \omega, \boldsymbol\gamma)=(\boldsymbol \omega', \boldsymbol\gamma')$.
Due  to Proposition \ref{Toep-def2},  we have
\begin{equation}\label{T}
\left<Te_{\boldsymbol\gamma },  e_{\boldsymbol\omega} \right>=
\frac{\tau_{(\boldsymbol\omega,\boldsymbol\gamma)}}{\tau_{(\boldsymbol \omega', \boldsymbol\gamma')}}\left<Te_{\boldsymbol\gamma'}, e_{\boldsymbol\omega'} \right>.
\end{equation}
The condition in the theorem implies that
  for any $\sigma\in \FF_{n_i}^+$, we   have
  \begin{equation}
  \label{bfrac}
  \sup_{\alpha\in \FF_{n_i}} \frac{b_{i, \sigma\alpha}^{(m_i)}}{b_{i,\alpha}}<\infty.
 \end{equation}
  On the other hand, we have
$$
\frac{\tau_{(\boldsymbol\omega,\boldsymbol\gamma)}}{\tau_{(\boldsymbol \omega', \boldsymbol\gamma')}}
= \left(\prod_{i=1}^k \sqrt{\frac{b^{(m_i)}_{i,\min\{\omega_i, \gamma_i\}}}{b^{(m_i)}_{i,\max\{\omega_i, \gamma_i\}}}}\right)
\left(\prod_{i=1}^k {\sqrt{b^{(m_i)}_{i,\max\{\boldsymbol \omega_i', \boldsymbol\gamma_i'\}}}}\right).
$$
Using  relation \eqref{bfrac}, we deduce that
$$
\liminf_{|\min \{\omega_i, \gamma_i\}|\to \infty} \frac{b^{(m_i)}_{i,\min\{\omega_i, \gamma_i\}}}{b^{(m_i)}_{i,\max\{\omega_i, \gamma_i\}}}>0,
$$
which implies
\begin{equation}
\label{TT}
\liminf_{|\min \{\omega_i, \gamma_i\}|\to \infty} \frac{\tau_{(\boldsymbol\omega,\boldsymbol\gamma)}}{\tau_{(\boldsymbol \omega', \boldsymbol\gamma')}}>0.
\end{equation}
Note that if  $ |\min \{\omega_i, \gamma_i\}|\to \infty$ for each $i\in \{1,\ldots, k\}$, then  $e_{\boldsymbol\gamma}\to 0$ weakly.  If $T$ is compact operator, then $Te_{\boldsymbol\gamma}\to 0$ in norm. Using relations
\eqref{T}  and \eqref{TT}, we deduce that
$\left<Te_{\boldsymbol\gamma'}, e_{\boldsymbol\omega'} \right>=0$. Now, using again relation \eqref{T}, we deduce that $\left<Te_{\boldsymbol\gamma },  e_{\boldsymbol\omega} \right>=0$
for any $(\boldsymbol \omega, \boldsymbol\gamma)\in \cC$ such that
${\bf s}(\boldsymbol \omega, \boldsymbol\gamma)=(\boldsymbol \omega', \boldsymbol\gamma')$.
Taking into account  Proposition \ref{Toep-def2}, we conclude that $T=0$.
The proof is complete.
\end{proof}
We should mention that the later theorem  was proved in \cite{Po-Toeplitz-poly-hyperball} in the particular case of the poly-hyperball.
In what follows we present another class of polydomains for which Theorem \ref{compact} holds.

Let ${\boldsymbol\varphi }:=(\varphi_1,\ldots, \varphi_k)$ be the $k$-tuple of positive regular free  holomorphic functions defined by
$$
\varphi_i=\sum_{\alpha\in \FF_{n_i}^+, |\alpha|\geq 1} Z_{i,\alpha}.
$$

\begin{corollary} If $T\in B(\otimes_{s=1}^k  F^2(H_{n_s}))$ is a compact weighted multi-Toeplitz operator with respect to the poldomain ${\bf D}_{\boldsymbol\varphi}^{\bf m}$, then $T=0$.
\end{corollary}
\begin{proof}
Let $\alpha\in \FF_{n_i}^+$ be such that $|\alpha|=d\in \{1,\ldots, k\}$. Note that  if $j\in \{1,\ldots, d\}$, then each tuple $(q_1,\ldots, q_j)\in \NN^k$ with $q_1+\cdots+q_j=d$ corresponds to a unique tuple
 $(\gamma_1,\ldots \gamma_j)$ with $\gamma_s\in \FF_{n_i}^+$ such that $\gamma_1\cdots \gamma_j=\alpha$ and $|\gamma_1|=q_1$,\ldots, $|\gamma_j|=q_j$.
On the other hand, it is well-known that
$$
\text{\rm card}\left\{ (q_1,\ldots, q_j)\in \NN^k: \  q_1+\cdots+q_j=d\right\}=
\left(\begin{matrix} d-1\\j-1\end{matrix}\right).
$$
Taking into account that
\begin{equation*}
 b_{i,\alpha}^{(m_i)}:= \sum_{j=1}^{|\alpha|}
\sum_{{\gamma_1,\ldots,\gamma_j\in \FF_{n_i}^+}\atop{{\gamma_1\cdots \gamma_j=\alpha }\atop {|\gamma_1|\geq
1,\ldots, |\gamma_j|\geq 1}}} a_{i,\gamma_1}\cdots a_{i,\gamma_j}\left(\begin{matrix} j+m_i-1\\m_i-1\end{matrix}\right)
   \qquad
\text{ if } \ \alpha\in \FF_{n_i}^+,  |\alpha|\geq 1,
\end{equation*}
and $a_{i,\alpha}=1$ for any $\alpha\in \FF_{n_i}^+$.
we deduce that
$$
b_{i,\alpha}^{(m_i)}=\sum_{j=1}^d \left(\begin{matrix} d-1\\j-1\end{matrix}\right)
\left(\begin{matrix} j+m_i-1\\m_i-1\end{matrix}\right).
$$
Consider the polynomial
$$p_d(x):= x^{m_i}(1+x)^d=\left(\begin{matrix} d\\0\end{matrix}\right)x^{m_i}+\left(\begin{matrix} d\\1\end{matrix}\right)x^{m_i+1}+\cdots + \left(\begin{matrix} d\\d\end{matrix}\right)x^{m_i+d}
$$
and note that the derivative of $p_d$ of order $m_i-1$  satisfies the relation
\begin{equation*}
\frac{p_d^{(m_i-1)}(1)}{(m_i-1) !}=\sum_{j=1}^{d +1}\left(\begin{matrix} d\\j-1\end{matrix}\right)
\left(\begin{matrix} j+m_i-1\\m_i-1\end{matrix}\right).
\end{equation*}
Consequently, we have
$$
\frac{b^{(m_i)}_{i,g_j^i\alpha }}{b^{(m_i)}_{i,\alpha}}=\frac{p_d^{(m_i-1)}(1)}{p_{d-1}^{(m_i-1)}(1)}.
$$
In order to apply Theorem \ref{compact},  we need to  prove that  $\frac{p_d^{(m_i-1)}(1)}{p_{d-1}^{(m_i-1)}(1)}$ is bounded  as $d\to \infty$.  First, note that if $m_i=1$, then
$\frac{p_d(1)}{p_{d-1}(1)}=2$. Now assume that $m_i\geq 2$.
A careful straightforward  computation leads to the fact
$$
p_d^{(m_i-1)} (1)= 2^{d-m_i+1}\left[ d(d-1)\cdots (d-m_i+2)+ q_{m_i-2}(d)\right],
$$
where $q_{m_i-2}$ is a polynomial of degree $m_i-2$ in the variable $d$. Then, we obtain
$$
\frac{p_d^{(m_i-1)}(1)}{p_{d-1}^{(m_i-1)}(1)}
=\frac{2^{d-m_i+1}\left[ d(d-1)\cdots (d-m_i+2)+ q_{m_i-2}(d)\right]}{2^{d-m_i}\left[ (d-1)(d-2)\cdots (d-m_i+1)+ q_{m_i-2}(d-1)\right]},
$$
which converges to $2$ as $d\to\infty$. Now, applying Theorem \ref{compact},  we can complete the proof.
\end{proof}

\bigskip

\section{Characterizations of weighted multi-Toeplitz operators associated with polydomains }

In this section, we  characterize the weighted multi-Toeplitz operators in terms of bounded free $k$-pluriharmonic functions on the radial part of ${\bf D_f^m}$ and  also  characterize the formal series
    which are  noncommutative Fourier representations  of weighted multi-Toeplitz operators.

The {\it noncommutative Berezin kernel} associated with any element
   ${\bf X}=\{X_{i,j}\}$ in the noncommutative polydomain ${\bf D_f^m}(\cH)$ is the operator
   $${\bf K_{f,X}}: \cH \to F^2(H_{n_1})\otimes \cdots \otimes  F^2(H_{n_k}) \otimes  \overline{{\bf \Delta_{f,X}^m}(I) (\cH)}$$
   defined by
   $$
   {\bf K_{f,X}}h:=\sum_{\beta_i\in \FF_{n_i}^+, i=1,\ldots,k}
   \sqrt{b_{1,\beta_1}^{(m_1)}}\cdots \sqrt{b_{k,\beta_k}^{(m_k)}}
   e^1_{\beta_1}\otimes \cdots \otimes  e^k_{\beta_k}\otimes {\bf \Delta_{f,X}^m}(I)^{1/2} X_{1,\beta_1}^*\cdots X_{k,\beta_k}^*h,
   $$
where   the defect operator is given  by
$$
{\bf \Delta_{f,X}^m}(I)  :=(id-\Phi_{f_1,X_1})^{m_1}\circ\cdots \circ (id-\Phi_{f_k,X_k})^{m_k}(I).
$$
The noncommutative Berezin kernel associated with a $k$-tuple
${\bf X}=({ X}_1,\ldots, { X}_k)$ in the noncommutative polydomain ${\bf D_f^m}(\cH)$ has the following properties.
\begin{enumerate}
\item[(i)] ${\bf K_{f,X}}$ is a contraction  and
$$
{\bf K_{f,X}^*}{\bf K_{f,X}}=
\lim_{q_k\to\infty}\ldots \lim_{q_1\to\infty}  (id-\Phi_{f_1,X_k}^{q_k})\circ\cdots \circ (id-\Phi_{f_k,X_1}^{q_1})(I).
$$
where the limits are in the weak  operator topology.
\item[(ii)]  If ${\bf X}$ is {\it pure},   i.e. $\Phi_{f_i,X_i}^p(I)\to 0$ strongly as $p\to \infty$, then
$${\bf K_{f,X}^*}{\bf K_{f,X}}=I_\cH. $$
\item[(iii)]  For any $i\in \{1,\ldots, k\}$ and $j\in \{1,\ldots, n_i\}$,  $${\bf K_{f,X}} { X}^*_{i,j}= ({\bf W}_{i,j}^*\otimes I)  {\bf K_{f,X}}.
    $$
\end{enumerate}
More on noncommutative Berezin transforms associated with noncommutative domains and polydomains can be found in  \cite{Po-poisson}, \cite{Po-domains-models},  \cite{Po-domains}, \cite{Po-Berezin2}, and  \cite{Po-Berezin1}.

The radial part of ${\bf D_f^m}$ is the noncommutative domain ${\bf D}_{{\bf f},rad}^{\bf m}$ whose representation on any Hilbert space $\cH$ is
$${\bf D}_{{\bf f},rad}^{\bf m}(\cH):= \cup_{r\in [0,1)}r{\bf D_f^m}(\cH).
$$
In what follows, we  also use the notation ${\bf X}_{\boldsymbol\alpha}:={\bf X}_{1,\alpha_1}\cdots {\bf X}_{k,\alpha_k}$  whenever $\boldsymbol\alpha:=(\alpha_1,\ldots, \alpha_k)\in {\bf F}_{\bf n}^+:=\FF_{n_1}^+\times\cdots \times \FF_{n_k}^+$, and   set $|\boldsymbol\alpha|:=|\alpha_1|+\cdots +|\alpha_k|$.

\begin{definition}We say that $F$ is a free  $k$-pluriharmonic function on the radial part of  ${\bf D_f^m}$  with coefficients in $B(\cK)$,  if its representation on a Hilbert space $\cH$ has the form
$$
F({\bf X})=\sum_{s_1\in \ZZ}\cdots \sum_{s_k\in \ZZ}
\sum_{(\boldsymbol \alpha, \boldsymbol \beta)\in \boldsymbol\cJ_{\bf s}} A_{(\boldsymbol \alpha, \boldsymbol \beta)}\otimes {\bf X}_{ \boldsymbol\alpha}{\bf X}_{\boldsymbol\beta}^*
$$
for any ${\bf X}\in {\bf D}_{{\bf f},rad}^{\bf m}(\cH)$, where the convergence is in the operator norm topology.
\end{definition}

An application  of the noncommutative  Berezin transforms associated with polydomains reveals  that $F$ is a free $k$-pluriharmonic function on the radial part of  ${\bf D_n^m}$  with coefficients in $B(\cK)$,  if and only if the series
$$
\sum_{s_1\in \ZZ}\cdots \sum_{s_k\in \ZZ}
\sum_{(\boldsymbol \alpha, \boldsymbol \beta)\in \boldsymbol\cJ_{\bf s}} r^{|\boldsymbol\alpha|+|\boldsymbol\beta|}A_{(\boldsymbol \alpha, \boldsymbol \beta)}\otimes {\bf W}_{\boldsymbol\alpha}{\bf W}_{\boldsymbol\beta}^*
$$
  convergences  in the operator norm topology for any $r\in [0,1)$.
A free  $k$-pluriharmonic function on the radial part of  ${\bf D_f^m}$  is called bounded if
$$\|F\|:=\sup_{{\bf X} \in {\bf D}_{{\bf f},rad}^{\bf m}(\cH)}||F({\bf X} )\|<\infty,
$$
where the supremum is taken over all Hilbert spaces $\cH$.

 In what follows, we need  the following result from \cite{Po-Toeplitz-poly-hyperball}.
\begin{lemma}\label{homo-decomp}
If $T\in B(\cK\bigotimes \otimes_{s=1}^k  F^2(H_{n_s}))$ and
   $\{T_{\bf s }\}_{{\bf s}=(s_1,\ldots, s_k)\in \ZZ^k}$ are the multi-homogeneous parts of $T$, then
   $$
   Tf=\lim_{N_1\to \infty}\ldots \lim_{N_k\to \infty} \sum_{{\bf s}=(s_1,\ldots, s_k)\in \ZZ^k, |s_j|\leq N_j}
   \left(1-\frac{|s_1|}{N_1+1}\right)\cdots \left(1-\frac{|s_k|}{N_k+1}\right) T_{\bf s }f
   $$
for any $f\in \cK\bigotimes \otimes_{s=1}^k  F^2(H_{n_s})$. Moreover,
$$
Tf=\lim_{N_1\to \infty}\ldots \lim_{N_k\to \infty} \sum_{{\bf s}=(s_1,\ldots, s_k)\in \ZZ^k, |s_j|\leq N_j} T_{\bf s }f
$$
for any $f\in \cK\otimes\cE_{p_1,\ldots, p_k}$  and any $(p_1,\ldots, p_k)\in \ZZ^k$.
\end{lemma}

Recall that, for each ${\bf s}=(s_1,\ldots, s_k)\in \ZZ^k$,
 $$
 \boldsymbol\cJ_{\bf s}:= \left\{ (\boldsymbol \alpha, \boldsymbol \beta):=(\alpha_1,\ldots,\alpha_k,\beta_1,\ldots, \beta_k): \   \alpha_i,\beta_i\in \FF_{n_i}^+  \text{ with } |\alpha_i|=s_i^+,   |\beta_i|=s_i^-\right\}
 $$
and  $\boldsymbol\cJ=\cup_{{\bf s}\in \ZZ^k} \boldsymbol\cJ_{\bf s}$.

The main result of this section is the following

\begin{theorem}\label{main}
If $T\in B(\cK\bigotimes \otimes_{s=1}^k  F^2(H_{n_s}))$, then the following statements are equivalent.
\begin{enumerate}
\item[(i)] $T$ is a  weighted multi-Toeplitz operator.
\item[(ii)] There is a   bounded free  $k$-pluriharmonic  function $F$ on  the radial polydomain ${\bf D}_{{\bf f},rad}^{\bf m}$ with coefficients in $B(\cK)$ such that
    $$T=\text{\rm SOT-}\lim_{r\to 1} F(r{\bf W}).$$
     \end{enumerate}
Moreover, the function $F$ is uniquely determined with the properties above, and
$$
\|T\|=\sup_{r\in [0,1)}\|F(r{\bf W})\|=\lim_{r\to 1}\|F(r{\bf W})\|=\sup_{\zeta\in \cP_\cK, \|\zeta\|\leq 1}\|F({\bf W})\zeta\|.
$$
\end{theorem}
\begin{proof}   Assume that $T$ is a weighted multi-Toeplitz operator and let
   $\{T_{\bf s}\}_{{\bf s}=(s_1,\ldots, s_k)\in \ZZ^k}$ be the multi-homogeneous parts of $T$.
    First, we need to show that  the multi-homogeneous parts of $T$ are also weighted multi-Toeplitz  operators.
Since $T\in B(\cK\bigotimes \otimes_{s=1}^k  F^2(H_{n_s}))$ is   a  weighted multi-Toeplitz,
  for any $ \boldsymbol\omega, \boldsymbol\gamma \in  {\bf F}_{n}^+$,
\begin{equation*}
\left<T(x\otimes e_{\boldsymbol\gamma }), y\otimes e_{\boldsymbol\omega} \right>
=\begin{cases}
\frac{\tau_{(\boldsymbol\omega,\boldsymbol\gamma)}}{\tau_{(\boldsymbol \omega', \boldsymbol\gamma')}}\left<T(x\otimes e_{\boldsymbol\gamma'}), y\otimes e_{\boldsymbol\omega'} \right>, & \text{ if  } (\boldsymbol \omega, \boldsymbol\gamma)\in \boldsymbol\cC,\\
0,&  \text{ if  } (\boldsymbol \omega, \boldsymbol\gamma)\in ({\bf F}_{\bf n }^+\times {\bf F}_{\bf n }^+)\backslash \boldsymbol\cC,
\end{cases}
\end{equation*}
where  $(\boldsymbol \omega', \boldsymbol\gamma'):={\bf s}(\boldsymbol \omega, \boldsymbol\gamma)$
 when $(\boldsymbol \omega, \boldsymbol\gamma)\in \boldsymbol\cC$.
Let $\boldsymbol\omega=(\omega_1,\ldots, \omega_k)$ , $\boldsymbol\gamma=(\gamma_1,\ldots, \gamma_k)$, $\boldsymbol\omega'=(\omega_1',\ldots, \omega_k')$ , $\boldsymbol\gamma'=(\gamma_1',\ldots, \gamma_k')$ be in ${\bf F}_{\bf n}^+$ such that  $(\boldsymbol \omega, \boldsymbol\gamma)\in \boldsymbol\cC$ and let
$(\boldsymbol \omega', \boldsymbol\gamma'):={\bf s}(\boldsymbol \omega, \boldsymbol\gamma)$.
Note that
$$
\Gamma (e^{i {\boldsymbol\theta}})^*e_{\boldsymbol\omega}=e^{-i|\omega_1|\theta_1}\cdots e^{-i|\omega_k|\theta_k}e_{\boldsymbol\omega},
$$
and $|\gamma_i|-|\omega_i|=|\gamma_i'|-|\omega_i'|$ for any $i\in \{1,\ldots, k\}$.
Using the relations above,
we deduce that
\begin{equation*}
\begin{split}
&\left<\left(I_\cK\otimes \Gamma (e^{i {\boldsymbol\theta}})\right) T
\left(I_\cK\otimes \Gamma (e^{i {\boldsymbol\theta}})\right)^*(x\otimes e_{\boldsymbol\gamma }), y\otimes e_{\boldsymbol\omega} \right>\\
&\qquad =
e^{-i(|\gamma_1|-|\omega_1|)\theta_1}\cdots e^{-i(|\gamma_k|-|\omega_k|)\theta_k}\left<T(x\otimes e_{\boldsymbol\gamma }), y\otimes e_{\boldsymbol\omega} \right>\\
&\qquad =
\frac{\tau_{(\boldsymbol\omega,\boldsymbol\gamma)}}{\tau_{(\boldsymbol \omega', \boldsymbol\gamma')}}
e^{-i(|\gamma_1'|-|\omega_1'|)\theta_1}\cdots e^{-i(|\gamma_k'|-|\omega_k'|)\theta_k}\left<T(x\otimes e_{\boldsymbol\gamma' }), y\otimes e_{\boldsymbol\omega'} \right>\\
&\qquad =
\frac{\tau_{(\boldsymbol\omega,\boldsymbol\gamma)}}{\tau_{(\boldsymbol \omega', \boldsymbol\gamma')}}
\left<\left(I_\cK\otimes \Gamma (e^{i {\boldsymbol\theta}})\right) T
\left(I_\cK\otimes \Gamma (e^{i {\boldsymbol\theta}})\right)^*(x\otimes e_{\boldsymbol\gamma' }), y\otimes e_{\boldsymbol\omega'} \right>
\end{split}
\end{equation*}
Consequently, we have
\begin{equation*}
\begin{split}
&\left<T_{\bf s}(x\otimes e_{\boldsymbol\gamma }), y\otimes e_{\boldsymbol\omega} \right>\\
&=
\left(\frac{1}{2\pi}\right)^k
\int_0^{2\pi}\cdots \int_0^{2\pi} e^{-is_1\theta_1}\cdots e^{-is_k\theta_k}\left<\left(I_\cK\otimes \Gamma (e^{i {\boldsymbol\theta}})\right) T
\left(I_\cK\otimes \Gamma (e^{i {\boldsymbol\theta}})\right)^*(x\otimes e_{\boldsymbol\gamma }), y\otimes e_{\boldsymbol\omega} \right> d\theta_1\ldots d\theta_k\\
&=
\frac{\tau_{(\boldsymbol\omega,\boldsymbol\gamma)}}{\tau_{(\boldsymbol \omega', \boldsymbol\gamma')}}
\left(\frac{1}{2\pi}\right)^k
\int_0^{2\pi}\cdots \int_0^{2\pi} e^{-is_1\theta_1}\cdots e^{-is_k\theta_k}\left<\left(I_\cK\otimes \Gamma (e^{i {\boldsymbol\theta}})\right) T
\left(I_\cK\otimes \Gamma (e^{i {\boldsymbol\theta}})\right)^*(x\otimes e_{\boldsymbol\gamma' }), y\otimes e_{\boldsymbol\omega'} \right> d\theta_1\ldots d\theta_k\\
&=
\frac{\tau_{(\boldsymbol\omega,\boldsymbol\gamma)}}{\tau_{(\boldsymbol \omega', \boldsymbol\gamma')}}
\left<T_{\bf s}(x\otimes e_{\boldsymbol\gamma' }), y\otimes e_{\boldsymbol\omega'} \right>.
\end{split}
\end{equation*}
On the other hand,  if $(\boldsymbol \omega, \boldsymbol\gamma)\in ({\bf F}_{\bf n }^+\times {\bf F}_{\bf n }^+)\backslash \boldsymbol\cC$, then similar calculations show that
$\left<T_{\bf s}(x\otimes e_{\boldsymbol\gamma }), y\otimes e_{\boldsymbol\omega} \right>=0$.
Consequently, $T_{\bf s}$ is a weighted multi-Toeplitz  operator.

    Now, we can apply  Theorem
\ref{multi-homo} to deduce that
\begin{equation*}
  T_{\bf s }g=q_{\bf s}({\bf W}, {\bf W}^*) g,\qquad g\in \cK\bigotimes \otimes_{s=1}^k  F^2(H_{n_s}),
\end{equation*}
where
 \begin{equation}
 \label{q}
 q_{\bf s}({\bf W}, {\bf W}^*):=
\sum_{(\boldsymbol \alpha, \boldsymbol \beta)\in \boldsymbol\cJ_{\bf s}} A_{(\boldsymbol \alpha, \boldsymbol \beta)}\otimes {\bf W}_{\boldsymbol\alpha }{\bf W}_{\boldsymbol\beta }^*,
\end{equation}
for some operators $A_{(\boldsymbol \alpha, \boldsymbol \beta)}\in B(\cK)$.
On the other hand, according to Lemma \ref{homo-decomp}, we have
   $$
Tf=\lim_{N_1\to \infty}\ldots \lim_{N_k\to \infty} \sum_{{\bf s}=(s_1,\ldots, s_k)\in \ZZ^k, |s_j|\leq N_j} T_{\bf s }f
$$
for any $f\in \cK\otimes\cE_{p_1,\ldots, p_k}$  and any $(p_1,\ldots, p_k)\in \ZZ^k$.
Denote by $\cP_\cK$ the linear span of all vectors of the form
$h\otimes e^1_{\alpha_1}\otimes \cdots\otimes e^k_{\alpha_k}$, where $h\in \cK$, $\alpha_i\in \FF_{n_i}^+$.
Combining the results above, we deduce that
\begin{equation}
\label {Tpp}
T\zeta=\lim_{N_1\to \infty}\ldots \lim_{N_k\to \infty} \sum_{{\bf s}=(s_1,\ldots, s_k)\in \ZZ^k, |s_j|\leq N_j} q_{\bf s}({\bf W}, {\bf W}^*)\zeta,\qquad \zeta\in \cP_\cK.
\end{equation}
We remark that, for any $x,y\in  \cK\bigotimes \otimes_{s=1}^k  F^2(H_{n_s})$, we have
\begin{equation*}
\begin{split}
|\left<T_{\bf s }x,y\right>|
&\leq
\left(\frac{1}{2\pi}\right)^k
\int_0^{2\pi}\cdots \int_0^{2\pi}  |\left<T
\left(I_\cK\otimes \Gamma (e^{i {\boldsymbol\theta}})^*\right) x, \left(I_\cK\otimes \Gamma (e^{i {\boldsymbol\theta}})^*\right)y\right>|d\theta_1\ldots d\theta_k\\
&\leq \|T\|\|x\|\|y\|,
\end{split}
\end{equation*}
which implies $\|T_{\bf s }\|\leq \|T\|$ for any ${\bf s}:=(s_1,\ldots, s_k)\in \ZZ^k$. As a consequence, we deduce that the series
\begin{equation*}
\begin{split}
\sum_{{\bf s} \in \ZZ^k}q_{\bf s}(r{\bf W}, r{\bf W}^*)&=
\sum_{{\bf s}=(s_1,\ldots, s_k)\in \ZZ^k}r^{|s_1|+\cdots +|s_k|}q_{\bf s }({\bf W}, {\bf W}^*)
=\sum_{{\bf s}=(s_1,\ldots, s_k)\in \ZZ^k}r^{|s_1|+\cdots +|s_k|} T_{\bf s}
\end{split}
\end{equation*}
are convergent in the operator norm topology.
Now, we prove that, for any $\zeta\in \cP_\cK$,
\begin{equation}
\label{conv}
\left\|\sum_{{\bf s} \in \ZZ^k}q_{\bf s}(r{\bf W}, r{\bf W}^*)\zeta-\sum_{{\bf s} \in \ZZ^k}q_{\bf s}({\bf W}, {\bf W}^*)\zeta \right\|\to 0,\quad \text{as } \ r\to 1.
\end{equation}
 It is enough to prove the relation when $\zeta\in \cK\otimes \cE_{p_1,\ldots, p_k}$.
To this end, let $\epsilon>0$ and note that    relation \eqref{Tpp}, implies the existence of   a finite set $\Gamma\subset  \ZZ^k$ such that
$\left\|\sum_{{\bf s}\in \ZZ^k\backslash \Gamma}q_{\bf s}({\bf W}, {\bf W}^*)\zeta \right\|<\epsilon$.
Since the vectors  $\{q_{\bf s }({\bf W}, {\bf W}^*)\zeta \}_{ {\bf s}\in \ZZ^k}$ are pairwise orthogonal
and
$$
\sum_{{\bf s} \in \ZZ^k\backslash \Gamma}\left\|q_{\bf s}(r{\bf W}, r{\bf W}^*)\zeta \right\|^2
=\sum_{{\bf s}
=(s_1,\ldots, s_k)\in \ZZ^k\backslash \Gamma}r^{(|s_1|+\cdots +|s_k|)}\left\|q_{\bf s}({\bf W}, {\bf W}^*)\zeta  \right\|^2\leq \epsilon^2,
$$
we deduce that
\begin{equation*}
\begin{split}
&\left\|\sum_{{\bf s}\in \ZZ^k}q_{\bf s}(r{\bf W}, r{\bf W}^*)\zeta -\sum_{ {\bf s}\in \ZZ^k}q_{\bf s}({\bf W}, {\bf W}^*)\zeta \right\|
\leq
\sum_{{\bf s}\in \Gamma}\left\|q_{\bf s}(r{\bf W}, r{\bf W}^*)\zeta-q_{\bf s}({\bf W}, {\bf W}^*)\zeta\right\|\\
&\qquad+
\left(\sum_{{\bf s}\in \ZZ^k\backslash \Gamma}\left\|q_{\bf s}(r{\bf W}, r{\bf W}^*)\zeta \right\|^2\right)^{1/2}
+\left(\sum_{ {\bf s}\in \ZZ^k\backslash \Gamma}\left\|q_{\bf s}({\bf W}, {\bf W}^*)\zeta \right\|^2\right)^{1/2}\\
& \qquad\leq
\sum_{{\bf s}\in \Gamma}\left\|q_{\bf s}(r{\bf W}, r{\bf W}^*)\zeta-q_{\bf s}({\bf W}, {\bf W}^*)\zeta\right\| + 2\epsilon.
\end{split}
\end{equation*}
Taking $r\to 1$, one can easily  deduce relation \eqref{conv}.
In what follows,  we prove that
\begin{equation}
\label{VN}
\left\|\sum_{{\bf s}\in \ZZ^k}q_{\bf s }(r{\bf W}, r{\bf W}^*)\right\|\leq \|T\|,\qquad r\in[0,1).
\end{equation}

We recall that  the noncommutative Berezin kernel associated with $r{\bf W}\in {\bf D_n^m}(\otimes _{i=1}^k F^2(H_{n_i}))$ is defined on
$\otimes_{i=1}^kF^2(H_{n_{i}})$ with values in $ \otimes_{i=1}^kF^2(H_{n_{i}})\otimes \cD_{r{\bf W}}\subset \left(\otimes_{i=1}^kF^2(H_{n_{i}})\right)\otimes \left(\otimes_{i=1}^kF^2(H_{n_{i}})\right)$, where $\cD_{r{\bf W}}:=\overline{\boldsymbol{\Delta}^{\bf m}_{r{\bf W}}(I)(\otimes_{i=1}^kF^2(H_{n_i}))}$.
Let $\boldsymbol\gamma=(\gamma_1,\ldots, \gamma_k)$ and $\boldsymbol \omega=(\omega_1,\ldots, \omega_k)$ be  $k$-tuples in ${\bf F}_{\bf n}^+$, set $q:=\max\{|\gamma_1|, \ldots |\gamma_k|, |\omega_1|,\ldots,|\omega_k|\}$, and define the operator
$$
\cT_q:=\sum_{s_1\in \ZZ, |s_1|\leq q}\cdots\sum_{s_k\in \ZZ, |s_k|\leq q}\sum_{(\boldsymbol \alpha, \boldsymbol \beta)\in \boldsymbol\cJ_{\bf s}} A_{(\boldsymbol \alpha, \boldsymbol \beta)}\otimes {\bf W}_{\boldsymbol\alpha} {\bf W}_{\boldsymbol\beta}^*,
$$
where we use the notation ${\bf W}_{\boldsymbol\alpha}:={\bf W}_{1,\alpha_1}\cdots {\bf W}_{k,\alpha_k}$ if $\boldsymbol\alpha:=(\alpha_1,\ldots, \alpha_k)\in{\bf F_n}^+:= \FF_{n_1}^+\times\cdots \times \FF_{n_k}^+$. We also set $e_{\boldsymbol\alpha}:=e_{\alpha_1}^1\otimes \cdots \otimes e_{\alpha_k}^k$ and $b_{\boldsymbol\alpha}^{({\bf m})}:=\sqrt{b_{1,\alpha_1}^{(m_1)}}\cdots \sqrt{b_{k,\alpha_k}^{(m_k)}}$.

Using the definition of the noncommutative Berezin kernel ${\bf K}_{r{\bf W}}$ and that of the universal model ${\bf W}$, we deduce that
\begin{equation*}
\begin{split}
&\left<(I_\cK\otimes {\bf K}_{r{\bf W}}^*) (T\otimes I_{\otimes_{i=1}^kF^2(H_{n_{i}})})(I_\cK\otimes {\bf K}_{r{\bf W}})(h\otimes e_{\boldsymbol\gamma}), h'\otimes e_{\boldsymbol\omega}\right>\\
&=\left<(T\otimes I_{\otimes_{i=1}^kF^2(H_{n_{i}})})
\sum_{\boldsymbol\alpha\in {\bf F}_{\bf n}^+}h\otimes b_{\boldsymbol\alpha}^{({\bf m})}e_{\boldsymbol\alpha}\otimes \boldsymbol\Delta_{r{\bf W}}(I)^{1/2}{\bf W}_{\boldsymbol\alpha}^*(e_{\boldsymbol\gamma})\right.,\\
&\qquad\qquad \left.
\sum_{\boldsymbol\beta\in {\bf F}_{\bf n}^+}h'\otimes b_{\boldsymbol\beta}^{({\bf m})}e_{\boldsymbol\beta}\otimes \boldsymbol\Delta_{r{\bf W}}(I)^{1/2}{\bf W}_{\boldsymbol\beta}^*(e_{\boldsymbol\omega})\right>
\\
&=
\sum_{\boldsymbol\alpha\in {\bf F}_{\bf n}^+}
\sum_{\boldsymbol\beta\in {\bf F}_{\bf n}^+}
\left<
T(h\otimes b_{\boldsymbol\alpha}^{({\bf m})}e_{\boldsymbol\alpha})\otimes \boldsymbol\Delta_{r{\bf W}}(I)^{1/2}{\bf W}_{\boldsymbol\alpha}^*(e_{\boldsymbol\gamma}),
h'\otimes b_{\boldsymbol\beta}^{({\bf m})}e_{\boldsymbol\beta}\otimes \boldsymbol\Delta_{r{\bf W}}(I)^{1/2}{\bf W}_{\boldsymbol\beta}^*(e_{\boldsymbol\omega})\right>
\\
&=
\sum_{\boldsymbol\alpha\in {\bf F}_{\bf n}^+}
\sum_{\boldsymbol\beta\in {\bf F}_{\bf n}^+}
\left<
T(h\otimes b_{\boldsymbol\alpha}^{({\bf m})}e_{\boldsymbol\alpha}), h'\otimes b_{\boldsymbol\beta}^{({\bf m})}e_{\boldsymbol\beta}\right>
\left<\boldsymbol\Delta_{r{\bf W}}(I)^{1/2}{\bf W}_{\boldsymbol\alpha}^*(e_{\boldsymbol\gamma}),
 \boldsymbol\Delta_{r{\bf W}}(I)^{1/2}{\bf W}_{\boldsymbol\beta}^*(e_{\boldsymbol\omega})\right> \\
 &=
 \sum_{s_1\in \ZZ, |s_1|\leq q}\cdots\sum_{s_k\in \ZZ, |s_k|\leq q}\sum_{{\alpha_i,\beta_i\in \FF_{n_k}^+, i\in \{1,\ldots, k\}}\atop{|\alpha_i|=s_i^+, |\beta_i|=s_i^-}}
 \left<
\cT_q(h\otimes b_{\boldsymbol\alpha}^{({\bf m})}e_{\boldsymbol\alpha}), h'\otimes b_{\boldsymbol\beta}^{({\bf m})}e_{\boldsymbol\beta}\right>\\
&\qquad \qquad \times
\left<\boldsymbol\Delta_{r{\bf W}}(I)^{1/2}{\bf W}_{\boldsymbol\alpha}^*(e_{\boldsymbol\gamma}),
 \boldsymbol\Delta_{r{\bf W}}(I)^{1/2}{\bf W}_{\boldsymbol\beta}^*(e_{\boldsymbol\omega})\right>
 \\
&=
 \sum_{\boldsymbol\alpha\in {\bf F}_{\bf n}^+}
\sum_{\boldsymbol\beta\in {\bf F}_{\bf n}^+}
\left<
\cT_q(h\otimes b_{\boldsymbol\alpha}^{({\bf m})}e_{\boldsymbol\alpha}), h'\otimes b_{\boldsymbol\beta}^{({\bf m})}e_{\boldsymbol\beta}\right>
\left<\boldsymbol\Delta_{r{\bf W}}(I)^{1/2}{\bf W}_{\boldsymbol\alpha}^*(e_{\boldsymbol\gamma}),
 \boldsymbol\Delta_{r{\bf W}}(I)^{1/2}{\bf W}_{\boldsymbol\beta}^*(e_{\boldsymbol\omega})\right>.
 \end{split}
\end{equation*}
On the other hand, note that the later sum is equal to

\begin{equation*}
\begin{split}
  &\left<(I_{\cK}\otimes {\bf K}_{r{\bf W}}^*)(\cT_q\otimes I_{\otimes_{i=1}^kF^2(H_{n_{i}})})
(I_{\cE}\otimes {\bf K}_{r{\bf W}})(h\otimes e_{\boldsymbol\gamma}), h'\otimes e_{\boldsymbol\omega}\right>\\
 &=
\sum_{s_1\in \ZZ, |m_1|\leq q}\cdots\sum_{s_k\in \ZZ, |m_k|\leq q}\sum_{{\alpha_i,\beta_i\in \FF_{n_k}^+, i\in \{1,\ldots, k\}}\atop{|\alpha_i|=s_i^+, |\beta_i|=s_i^-}}\\
&\qquad \qquad
\left< \left(A_{(\alpha_1,\ldots, \alpha_k, \beta_1,\ldots, \beta_k)}\otimes r^{\sum_{i=1}^k(|\alpha_i|+|\beta_i|)} {\bf W}_{\boldsymbol\alpha} {\bf W}_{\boldsymbol\beta}^*\right)(h\otimes e_{\gamma}), h'\otimes e_{\boldsymbol\omega}\right>\\
&=\left<\sum_{ {\bf s}\in \ZZ^k}q_{ \bf s}(r{\bf W}, r{\bf W}^*)(h\otimes e_{\boldsymbol\gamma}), h'\otimes e_{\boldsymbol\omega}\right>.
\end{split}
\end{equation*}
Hence, we deduce that
\begin{equation*}
(I_\cK\otimes {\bf K}_{r{\bf W}}^*) (T\otimes I_{\otimes_{i=1}^kF^2(H_{n_{i}})})(I_\cK\otimes {\bf K}_{r{\bf W}})
=\sum_{ {\bf s}\in \ZZ^k}q_{\bf s }(r{\bf W}, r{\bf W}^*)
\end{equation*}
for any $r\in [0,1)$. Since  the noncommutative Berezin kernel
 ${\bf K}_{r{\bf W}}$ is an isometry, we obtain inequality \eqref{VN}.
In what follows, we prove that
\begin{equation}\label{SOT}
\text{\rm SOT-}\lim_{r \to 1}\sum_{ {\bf s}\in \ZZ^k}q_{\bf s }(r{\bf W}, r{\bf W}^*)=T.
\end{equation}
Indeed, let  $\epsilon>0$  and $\psi\in \cK\bigotimes \otimes_{s=1}^k  F^2(H_{n_s})$. Then there is $\zeta\in \cP_\cK$ such that $\|\psi-\zeta\|<\epsilon$ and, due to relations \eqref{Tpp} and \eqref{VN}, we deduce that
\begin{equation*}
\begin{split}
&\left\|\sum_{ {\bf s}\in \ZZ^k}q_{\bf s }(r{\bf W}, r{\bf W}^*)\psi-T\psi\right\|
\leq
\left\|\sum_{ {\bf s}\in \ZZ^k}q_{\bf s }(r{\bf W}, r{\bf W}^*)(\psi-\zeta)\right\|\\
&\qquad +\left\|\sum_{ {\bf s}\in \ZZ^k}q_{\bf s }(r{\bf W}, r{\bf W}^*)\zeta-\sum_{ {\bf s}\in \ZZ^k}q_{\bf s }({\bf W}, {\bf W}^*)\zeta\right\| +
\left\|\sum_{ {\bf s}\in \ZZ^k}q_{\bf s }({\bf W}, {\bf W}^*)\zeta-T\psi\right\|
\\
&\qquad \leq 2\|T\|\|\psi-\zeta\| +\left\|\sum_{ {\bf s}\in \ZZ^k}q_{\bf s }(r{\bf W}, r{\bf W}^*)\zeta-\sum_{ {\bf s}\in \ZZ^k}q_{\bf s}({\bf W}, {\bf W}^*)\zeta\right\|\\
&\qquad \leq 2\|T\|\epsilon +\left\|\sum_{ {\bf s}\in \ZZ^k}q_{\bf s }(r{\bf W}, r{\bf W}^*)\zeta-\sum_{ {\bf s}\in \ZZ^k}q_{\bf s}({\bf W}, {\bf W}^*)\zeta\right\|
\end{split}
\end{equation*}
Consequently, using relation \eqref{conv}, we deduce that
$$
\limsup_{r\to 1} \left\|\sum_{ {\bf s}\in \ZZ^k}q_{\bf s }(r{\bf W}, r{\bf W}^*)\psi-T\psi\right\|\leq 2\|T\|\epsilon
$$
for any $\epsilon >0$. Hence, we can see that
 relation \eqref{SOT} holds.
Note that $F({\bf X}):=\sum_{ {\bf s}\in \ZZ^k}q_{\bf s }({\bf X}, {\bf X}^*) $ is a free  $k$-pluriharmonic function on the radial polydomain  ${\bf D}_{{\bf f}, rad}^{\bf m}$.

  Now, we prove the converse (ii)$\implies$(i).
To this end, assume that there is a free $k$-pluriharmonic  function $F$ on  the radial polydomain ${\bf D}_{{\bf f},rad}^{\bf m}$ with coefficients in $B(\cK)$  of the form
$F({\bf X})=\sum_{ {\bf s}\in \ZZ^k}q_{ \bf s}({\bf X}, {\bf X}^*)$ such that
    $T=\text{\rm SOT-}\lim_{r\to 1} F(r{\bf W}).$
    Consequently,
    $$
    F(r{\bf W})=\sum_{ {\bf s}\in \ZZ^k}q_{ \bf s}(r{\bf W}, r{\bf W}^*),
    $$
    where the convergence of the series is in the operator norm topology and $q_{\bf s }(r{\bf W}, r{\bf W})^*$ has the form described by relation \eqref{q}.
Note that $ q_{\bf s }({\bf W}, {\bf W}^*)$ is a weighted multi-Toeplitz operator. Hence, we deduce that $ F(r{\bf W})$ is weighted  multi-Toeplitz operator  as well. Taking into account that the set of weighted multi-Toeplitz operators is WOT-closed  and  the fact that  $T=\text{\rm SOT-}\lim_{r\to 1} F(r{\bf W})$,  we conclude that   $T$ a  weighted multi-Toeplitz operator. Therefore condition (i) holds.

Assume that $F$ and $G$ are bounded free $k$-pluriharmonic functions  on the radial polydomain
${\bf D}_{{\bf f},rad}^{\bf m}$ with coefficients in $B(\cK)$ such that
    $$T=\text{\rm SOT-}\lim_{r\to 1} F(r{\bf W})=\text{\rm SOT-}\lim_{r\to 1} G(r{\bf W}).$$
Let
$F({\bf X}):=\sum_{(\boldsymbol \alpha, \boldsymbol \beta)\in \boldsymbol \cJ}A_{   (\boldsymbol \alpha, \boldsymbol \beta)}\otimes {\bf X}_{\boldsymbol\alpha}{\bf X}_{\boldsymbol\beta}^*$ and
$G({\bf X}):=\sum_{(\boldsymbol \alpha, \boldsymbol \beta)\in \boldsymbol \cJ}C_{   (\boldsymbol \alpha, \boldsymbol \beta)}\otimes {\bf X}_{\boldsymbol\alpha}{\bf X}_{\boldsymbol\beta}^*$ be the representations of $F$ and $G$, respectively.
If $ \boldsymbol\omega, \boldsymbol\gamma \in {\bf F}_{\bf n}^+$ and $x,y\in \cK$, we have
\begin{equation*}
\begin{split}
\left<T(x\otimes e_{\boldsymbol\gamma }), y\otimes e_{\boldsymbol\omega} \right>
&=\lim_{r\to 1}\left<F(r{\bf W})(x\otimes e_{\boldsymbol\gamma }), y\otimes e_{\boldsymbol\omega} \right>\\
&=
\lim_{r\to 1}\begin{cases}
\tau_{(\boldsymbol\omega,\boldsymbol\gamma)}r^{|{\bf s}(\boldsymbol\omega,\boldsymbol\gamma)|}\left<A_{{\bf s}(\boldsymbol \omega, \boldsymbol\gamma)}x,y\right>, & \text{ if  } (\boldsymbol \omega, \boldsymbol\gamma)\in \boldsymbol\cC,\\
0,&  \text{ if  } (\boldsymbol \omega, \boldsymbol\gamma)\in ({\bf F}_{\bf n }^+\times {\bf F}_{\bf n }^+)\backslash \boldsymbol\cC
\end{cases}\\
&= \begin{cases}
\tau_{(\boldsymbol\omega,\boldsymbol\gamma)}
\left<A_{{\bf s}(\boldsymbol \omega, \boldsymbol\gamma)}x,y\right>, & \text{ if  } (\boldsymbol \omega, \boldsymbol\gamma)\in \boldsymbol\cC,\\
0,&  \text{ if  } (\boldsymbol \omega, \boldsymbol\gamma)\in ({\bf F}_{\bf n }^+\times {\bf F}_{\bf n }^+)\backslash \boldsymbol\cC
\end{cases}
\end{split}
\end{equation*}
Similarly, we can prove that

\begin{equation*}
\begin{split}
\left<T(x\otimes e_{\boldsymbol\gamma }), y\otimes e_{\boldsymbol\omega} \right>=
\begin{cases}
\tau_{(\boldsymbol\omega,\boldsymbol\gamma)}
\left<C_{{\bf s}(\boldsymbol \omega, \boldsymbol\gamma)}x,y\right>, & \text{ if  } (\boldsymbol \omega, \boldsymbol\gamma)\in \boldsymbol\cC,\\
0,&  \text{ if  } (\boldsymbol \omega, \boldsymbol\gamma)\in ({\bf F}_{\bf n }^+\times {\bf F}_{\bf n }^+)\backslash \boldsymbol\cC.
\end{cases}
\end{split}
\end{equation*}
Now, it is clear that
$A_{{\bf s}(\boldsymbol \omega, \boldsymbol\gamma)}=C_{{\bf s}(\boldsymbol \omega, \boldsymbol\gamma)}$, which proves that $F=G$.

To prove the last part of the theorem, note that relation \eqref{Tpp} implies $\|T\|=\sup_{\zeta\in \cP_\cK, \|\zeta\|\leq 1}\|F({\bf W})\|$.  Taking into account relations \eqref{VN} and \eqref{SOT}, we deduce that
$\|T\|=\sup_{r\in [0,1)}\|F(r{\bf W})\|$. On the other hand, using the noncommutative Berezin kernel, we have

$$ F(r_1{\bf W})=(I_\cK\otimes {\bf K}_{\frac{r_1}{r_2}{\bf W}}^*) (F(r_2{\bf W})\otimes I_{\otimes_{i=1}^kF^2(H_{n_{i}})})(I_\cK\otimes {\bf K}_{\frac{r_1}{r_2}{\bf W}})
$$
for any $0\leq r_1<r_2<1$. Since ${\bf K}_{\frac{r_1}{r_2}{\bf W}}^*$ is an isomery, we deduce that
$\|F(r_1{\bf W})\|\leq \|F(r_2{\bf W})\|$. Now, it is clear  that
$\sup_{r\in [0,1)}\|F(r{\bf W})\|=\lim_{r\to 1}\|F(r{\bf W})\|$.
The proof is complete.
\end{proof}

\begin{corollary} \label{norm} If $T\in B(\cK\bigotimes \otimes_{s=1}^k  F^2(H_{n_s}))$, then the following statements are equivalent.
\begin{enumerate}
\item[(i)] $T$ is a multi-Toeplitz operator.
    \item[(ii)]
     $T\in\text{\rm span}\left\{ C\otimes {\bf W}_{\boldsymbol\alpha}{\bf W}_{\boldsymbol\beta}^* :\   C\in B(\cK),   (\boldsymbol \alpha, \boldsymbol \beta)\in \boldsymbol\cJ\right\}^{-\text{\rm SOT}}$.
      \item[(iii)]
     $T\in\text{\rm span}\left\{ C\otimes {\bf W}_{\boldsymbol\alpha}{\bf W}_{\boldsymbol\beta}^* :\   C\in B(\cK),   (\boldsymbol \alpha, \boldsymbol \beta)\in \boldsymbol\cJ\right\}^{-\text{\rm WOT}}$.
     \end{enumerate}
 \end{corollary}

 If $ (\boldsymbol \alpha, \boldsymbol \beta)\in {\bf F}_{\bf n}^+\times {\bf F}_{\bf n}^+$, we define its length to be
$| (\boldsymbol \alpha, \boldsymbol \beta)|:=|\alpha_1|+\cdots +|\alpha_k|+|\beta_1|+\cdots +|\beta|_k$. Similarly, if ${\bf s}=(s_1,\ldots, s_k)\in \ZZ^k$, we set $|{\bf s}|:=|s_1|+\cdots +|s_k|$.

\begin{theorem}   \label{Fourier} Let $\{A_{   (\boldsymbol \alpha, \boldsymbol \beta)}\}_{(\boldsymbol \alpha, \boldsymbol \beta)\in \boldsymbol\cJ}$ be  a family of  operators on the Hilbert space $\cK$ and let
$$
\varphi({\bf W}, {\bf W}^*):=\sum_{(\boldsymbol \alpha, \boldsymbol \beta)\in \boldsymbol\cJ}A_{   (\boldsymbol \alpha, \boldsymbol \beta)}\otimes {\bf W}_{\boldsymbol\alpha}{\bf W}_{\boldsymbol\beta}^*=\sum_{{\bf t}\in\ZZ^k}q_{\bf t}({\bf W}, {\bf W}^*)
$$
be a formal series.
Then $\varphi({\bf W}, {\bf W}^*)$ is the Fourier representation of a weighted multi-Toeplitz operator  if and only if,    for each $\boldsymbol\gamma\in {\bf F}_{\bf n}^+$, the series
$$
\sum_{\boldsymbol\omega\in {\bf F}_{\bf n}^+: (\boldsymbol\omega,\boldsymbol\gamma)\in \boldsymbol\cC}
\tau_{(\boldsymbol\omega,\boldsymbol\gamma)}^2 A_{{\bf s}(\boldsymbol\omega,\boldsymbol\gamma)}^*
A_{{\bf s}(\boldsymbol\omega,\boldsymbol\gamma)}\qquad \text{is {\rm WOT}-convergent}
$$
 and
$$\sup_{r\in [0,1)} \sup_{\zeta\in \cP_\cK, \|\zeta\|\leq 1} \|\varphi(r{\bf W}, r{\bf W}^*)\zeta\|<\infty.$$

In this case,
$$
T=\text{\rm SOT-}\lim_{r\to 1} \varphi (r{\bf W}, r{\bf W}^*)\quad \text{and}\quad \|T\|=\sup_{r\in [0,1)}   \|\varphi(r{\bf W}, r{\bf W}^*)\|.
$$
\end{theorem}
\begin{proof}
Assume that $\varphi({\bf W}, {\bf W}^*)$ is the  Fourier representation of a weighted multi-Toeplitz operator  $T$.
Then $
T\zeta=\varphi_T({\bf W}, {\bf W}^*)\zeta$ for any  $\zeta\in \cP_K$.
Note that  the multi-homogeneous parts of $T$, i.e  $\{T_{\bf s}\}_{{\bf s}=(s_1,\ldots, s_k)\in \ZZ^k}$ satisfy the following relations:
 \begin{equation*}
 \begin{split}
 T_{\bf s} f&=(I_\cK\otimes {\bf P}_{s_1+p_1+\cdots +s_k+p_k})Tf\\
 &=(I_\cK\otimes {\bf P}_{s_1+p_1+\cdots +s_k+p_k})\varphi({\bf W}, {\bf W}^*)f\\
 &=  (I_\cK\otimes {\bf P}_{s_1+p_1+\cdots +s_k+p_k})
    \sum_{{\bf t}\in \ZZ^k}\sum_{(\boldsymbol \alpha, \boldsymbol \beta)\in \boldsymbol\cJ_{\bf t}}(A_{   (\boldsymbol \alpha, \boldsymbol \beta)}\otimes {\bf W}_{\boldsymbol\alpha}{\bf W}_{\boldsymbol\beta}^*)f
 \end{split}
 \end{equation*}
 for any $f\in \cK\otimes \cE_{p_1,\ldots, p_k}$. Consequently, taking into account that
 $$
 \sum_{(\boldsymbol \alpha, \boldsymbol \beta)\in \boldsymbol\cJ_{\bf t}}(A_{   (\boldsymbol \alpha, \boldsymbol \beta)}\otimes {\bf W}_{\boldsymbol\alpha}{\bf W}_{\boldsymbol\beta}^*)f\in \cK\otimes \cE_{t_1+p_1,\ldots,t_k+ p_k},
 $$
  we conclude that
  $$
  T_{\bf s} f=\sum_{(\boldsymbol \alpha, \boldsymbol \beta)\in \boldsymbol\cJ_{\bf s}}(A_{   (\boldsymbol \alpha, \boldsymbol \beta)}\otimes {\bf W}_{\boldsymbol\alpha}{\bf W}_{\boldsymbol\beta}^*)f,\qquad  f\in \cK\otimes \cE_{p_1,\ldots, p_k}.
  $$
  Since
$$
 \otimes_{s=1}^k F^2(H_{n_s})=\bigoplus_{(p_1,\ldots, p_k)\in \ZZ^k} \cE_{p_1,\ldots, p_k},
 $$
 we deduce that $T_{\bf s}=\sum_{(\boldsymbol \alpha, \boldsymbol \beta)\in \boldsymbol\cJ_{\bf s}}A_{   (\boldsymbol \alpha, \boldsymbol \beta)}\otimes {\bf W}_{\boldsymbol\alpha}
 {\bf W}_{\boldsymbol\beta}^*=q_{\bf s}({\bf W}, {\bf W}^*)$ for any ${\bf s}=(s_1,\ldots, s_k)\in \ZZ^+$.

According to the proof of Theorem \ref{main}, we have
$\|T_{\bf s}\|\leq \|T\|$ for any ${\bf s}\in \ZZ^+$ and, as a consequence, the series
$\sum_{{\bf s}=(s_1,\ldots, s_k)\in \ZZ^+}r^{|s_1|+\cdots +|s_k|} T_{\bf s}$ is convergent in the operator norm topology.
Moreover, according to inequality \eqref{VN}, we have
\begin{equation*}
\left\|\sum_{{\bf s}\in \ZZ^k}q_{\bf s}(r{\bf W}, r{\bf W}^*)\right\|\leq \|T\|,\qquad r\in[0,1),
\end{equation*}
which implies
$\sup_{r\in [0,1)}\|\varphi (r{\bf W}, r{\bf W}^*)\zeta\|<\infty$ for any $\zeta\in \cP_\cK$.
 Moreover, in the proof of Theorem \ref{main}, we also proved that
$
T=\text{\rm SOT-}\lim_{r\to 1} \varphi (r{\bf W}, r{\bf W}^*).
$
On the other hand, Theorem \ref{main} shows that  $\|T\|=\sup_{r\in [0,1)}\|\varphi (r{\bf W}, r{\bf W}^*)\|$.
Consequently,
$\sup_{r\in [0,1)}\|\varphi (r{\bf W}, r{\bf W}^*)(h\otimes e_\gamma)\|<\infty$ for each $\gamma\in {\bf F}_{\bf n}^+$ and $h\in \cK$. Since
$$
\varphi (r{\bf W}, r{\bf W}^*)(x\otimes e_\gamma)
=
\sum_{\boldsymbol\omega\in {\bf F}_{\bf n}^+: (\boldsymbol\omega,\boldsymbol\gamma)\in \boldsymbol\cC} r^{2|{\bf s}(\boldsymbol \omega, \boldsymbol\gamma)|}
\tau_{(\boldsymbol\omega,\boldsymbol\gamma)}^2\|
\|A_{{\bf s}(\boldsymbol\omega,\boldsymbol\gamma)}x\|^2,
$$
we deduce that
$$
\sum_{\boldsymbol\omega\in {\bf F}_{\bf n}^+: (\boldsymbol\omega,\boldsymbol\gamma)\in \boldsymbol\cC}
\tau_{(\boldsymbol\omega,\boldsymbol\gamma)}^2 A_{{\bf s}(\boldsymbol\omega,\boldsymbol\gamma)}^*
A_{{\bf s}(\boldsymbol\omega,\boldsymbol\gamma)}\qquad \text{is {\rm WOT}-convergent}.
$$
Due to the fact that
$$\sup_{r\in [0,1)} \sup_{\zeta\in \cP_\cK, \|\zeta\|\leq 1} \|\varphi(r{\bf W}, r{\bf W}^*)\zeta\|
=\sup_{r\in [0,1)}\|\varphi (r{\bf W}, r{\bf W}^*)\|
<\infty,$$
the direct implication is proved.

Now, we prove the converse.   Assume that, for each $\boldsymbol\gamma\in {\bf F}_{\bf n}^+$, the series
$$
\sum_{\boldsymbol\omega\in {\bf F}_{\bf n}^+: (\boldsymbol\omega,\boldsymbol\gamma)\in \boldsymbol\cC}
\tau_{(\boldsymbol\omega,\boldsymbol\gamma)}^2 A_{{\bf s}(\boldsymbol\omega,\boldsymbol\gamma)}^*
A_{{\bf s}(\boldsymbol\omega,\boldsymbol\gamma)}\qquad \text{is {\rm WOT}-convergent}
$$
 and
$$\sup_{r\in [0,1)} \sup_{\zeta\in \cP_\cK, \|\zeta\|\leq 1} \|\varphi(r{\bf W}, r{\bf W}^*)\zeta\|<\infty.$$
As in the proof of Theorem \ref{multi-homo}, it is easy to see that, for any $h\in\cK$,
$$
 \sum_{\boldsymbol\omega\in \cF: (\boldsymbol\omega, \boldsymbol\gamma)\in \boldsymbol\cC}
 \tau_{(\boldsymbol\omega,\boldsymbol\gamma)}A_{{\bf s}(\boldsymbol \omega, \boldsymbol\gamma)}h \otimes e_{\boldsymbol\omega }
$$
is a vector in $\cK\bigotimes \otimes_{i=1}^k F^2(H_{n_i})$ and, therefore, so are  $ \varphi ({\bf W}, {\bf W}^*)\zeta$ and $ \varphi (r{\bf W}, r{\bf W}^*)\zeta$ for any $\zeta\in \cP_\cK$, $r\in [0,1)$.
Hence, we deduce that
\begin{equation}
\label{frr}
 \lim_{r\to 1}\varphi (r{\bf W}, r{\bf W}^*)\zeta= \varphi ({\bf W}, {\bf W}^*)\zeta,\qquad \zeta\in \cP_\cK.
\end{equation}
Using the fact that, for each $r\in [0,1)$,
 \begin{equation*}
 \sup_{\zeta\in\cP_\cK, \|\zeta\|\leq 1}\|\varphi (r{\bf W}, r{\bf W}^*)\zeta\|<\infty,
\end{equation*}
we conclude that there is a bounded linear operator $D_r\in B(\cK\bigotimes \otimes_{i=1}^k F^2(H_{n_i}))$ such that
\begin{equation}
\label{Trf}
  D_r \zeta=\varphi (r{\bf W}, r{\bf W}^*)\zeta,\qquad \zeta\in \cP_\cK.
\end{equation}
Note that, for any $ \boldsymbol\omega, \boldsymbol\gamma \in {\bf F}_{n}^+$ and $h,\ell\in \cK$,
\begin{equation*}
\begin{split}
\left<T(h\otimes e_{\boldsymbol\gamma }), \ell\otimes e_{\boldsymbol\omega} \right>
&=
\left<\varphi (r{\bf W}, r{\bf W}^*)(h\otimes e_{\boldsymbol\gamma }), \ell\otimes e_{\boldsymbol\omega} \right>\\
&=\begin{cases}
\tau_{(\boldsymbol\omega,\boldsymbol\gamma)}\left<r^{|{\bf s}(\boldsymbol \omega, \boldsymbol\gamma)|}A_{{\bf s}(\boldsymbol \omega, \boldsymbol\gamma)}h,\ell\right>, & \text{ if  } (\boldsymbol \omega, \boldsymbol\gamma)\in \boldsymbol\cC,\\
0,&  \text{ if  } (\boldsymbol \omega, \boldsymbol\gamma)\in ({\bf F}_{\bf n }^+\times {\bf F}_{\bf n }^+)\backslash \boldsymbol\cC.
\end{cases}
\end{split}
\end{equation*}
Consequently, $D_r$ is a weighted multi-Toeplitz operator.
Now, note that due to relation \eqref{frr} and the fact that  $\sup_{ r\in[0,1)} \sup_{\zeta\in\cP_\cK, \|\zeta\|\leq 1}\|\varphi (r{\bf W}, r{\bf W}^*)\zeta\|<\infty$, we deduce that
$$\sup_{\zeta\in \cP_\cK, \|\zeta\|\leq 1}\|\varphi ({\bf W}, {\bf W}^*)\zeta\|<\infty.
$$
Consequently,  there is a bounded linear operator $T$ on
$\cK\bigotimes \otimes_{i=1}^k F^2(H_{n_i})$ such that $T\zeta=
\varphi ({\bf W}, {\bf W}^*)\zeta $ for any $\zeta\in \cP_\cK$.
Now, it is clear that
$$
\lim_{r\to 1}D_r \zeta=\lim_{r\to 1} \varphi (r{\bf W}, r{\bf W}^*)\zeta=
\varphi ({\bf W}, {\bf W}^*)\zeta=T\zeta
$$
for any $\zeta\in \cP_\cK$ and  $\sup_{r\in [0,1)}\|D_r\|<\infty$. This implies $T=\text{\rm SOT-}\lim_{r\to 1}D_r$.  Since $D_r$ is a weighted multi-Toeplitz operator and the set of all weighted multi-Toeplitz operators is WOT-closed, we  deduce that    $T$ is a  weighted multi-Toeplitz operator.
Since $T\zeta=\varphi ({\bf W}, {\bf W}^*)\zeta$ for any $\zeta\in \cP_\cK$, Corollary \ref{formal2} shows that
$ \varphi ({\bf W}, {\bf W}^*)$ is the formal Fourier representation of $T$.
 The proof is complete.
\end{proof}

\bigskip

\section{Free $k$-pluriharmonic functions on polydomains}

The results for the previous sections are used to prove that the   bounded free $k$-pluriharmonic  functions on the radial polydomain ${\bf D}_{{\bf f},rad}^{\bf m}$
  are precisely those  that are noncommutative Berezin transforms of the weighted multi-Toeplitz operators. In this setting, we solve the Dirichlet extension problem.

Denote by  ${\bf PH}_\cK^\infty({\bf D}_{{\bf f},rad}^{\bf m})$  the set of all bounded free
 $k$-pluriharmonic functions on the radial  polydomain  ${\bf D}_{{\bf n},rad}^{\bf m}$ with coefficients in
$B(\cK)$.
We define the norms $\|\cdot
\|_m:M_m\left({\bf PH}_\cK^\infty({\bf D}_{{\bf f},rad}^{\bf m})\right)\to [0,\infty)$, $m\in \NN$,  by
setting
$$
\|[F_{ij}]_m\|_m:= \sup \|[F_{ij}({\bf X})]_m\|,
$$
where the supremum is taken over all elements ${\bf X}\in {\bf D}_{{\bf f},rad}^{\bf m}(\cH)$ and any Hilbert space $\cH$. It is easy to see that the norms
$\|\cdot\|_m$, $m\in\NN$, determine  an operator space
structure  on ${\bf PH}_\cK^\infty({\bf D}_{{\bf f},rad}^{\bf m})$,
 in the sense of Ruan (see e.g. \cite{ER}).

The   {\it extended noncommutative Berezin transform at} ${\bf X}\in {\bf D}_{{\bf f},rad}^{\bf m}(\cH)$
 is the map
 $$\widetilde{\bf B}_{\bf X}: B(\cK \bigotimes  \otimes_{i=1}F^2(H_{n_i}))\to B(\cK)\otimes_{min} B(\cH)
 $$
 defined by
 \begin{equation*}
 \widetilde{\bf B}_{\bf X}[g]:= \left(I_\cK\otimes {\bf K}_{\bf X} ^*\right) (g\otimes I_\cH)\left(I_\cK\otimes{\bf K}_{\bf X}\right),
 \quad g\in B(\cK\bigotimes \otimes_{i=1}^k F^2(H_{n_i})),
 \end{equation*}
where  ${\bf K}_{\bf X}:\cH \to \left(\otimes_{i=1}^kF^2(H_{n_i})\right)\otimes
\cH$  is noncommutative Berezin kernel associated with ${\bf X}\in {\bf D}_{{\bf f},rad}^{\bf m}(\cH)$.

Throughout this section, we assume that $\cH$ is a separable infinitely dimensional Hilbert space. Consequently, one can identify any free $k$-pluriharmonic function with its representation on $\cH$.
Let  $\boldsymbol{\cT }$ be the set of all  of all weighted multi-Toeplitz operators  on $\cK\otimes\bigotimes_{i=1}^k F^2(H_{n_i})$.
As in \cite{Po-Toeplitz-poly-hyperball},  one can proves the following characterization of
bounded  free $k$-pluriharmonic  functions on the radial polydomain ${\bf D}_{{\bf f},rad}^{\bf m}$.

 \begin{theorem}\label{bounded}
  If $F: {\bf D}_{{\bf f},rad}^{\bf m}(\cH)\to B(\cK)\otimes_{min}B(\cH)$, then the following statements are equivalent.
\begin{enumerate}
\item[(i)] $F$ is a bounded free $k$-pluriharmonic function.
\item[(ii)]
There exists $T\in \boldsymbol{\cT }$ such
that
$$F({\bf X})=  \widetilde{\bf B}_{\bf X}[T], \qquad
{\bf X}\in {\bf D}_{{\bf f},rad}^{\bf m}(\cH).
$$
\end{enumerate}
In this case,
  $T=\text{\rm SOT-}\lim\limits_{r\to 1}F(r{\bf W}).
  $
   Moreover, the map
$$
\Phi:{\bf PH}_\cK^\infty({\bf D}_{{\bf f},rad}^{\bf m})\to \boldsymbol{\cT}\quad
\text{ defined by } \quad \Phi(F):=T
$$ is a completely   isometric isomorphism of operator spaces.
\end{theorem}
\begin{proof}  If item (i)  holds,  Theorem \ref{main} shows that $T:=\text{\rm SOT-}\lim\limits_{r\to 1}F(r{\bf W})
  $
is a weighted multi-Toeplitz operator and $\|T\|=\sup_{r\in [0,1)}\|F(r{\bf W})\|$. On the other hand, we have
$$
F(r{\bf X})= \widetilde{\bf B}_{\bf X}[F(r{\bf W})], \qquad
{\bf X}\in {\bf D}_{{\bf f},rad}^{\bf m}(\cH).
$$
Taking $r\to 1$, we deduce that $F({\bf X})=  \widetilde{\bf B}_{\bf X}[T]$ for
${\bf X}\in {\bf D}_{{\bf f},rad}^{\bf m}(\cH).
$

If item (ii) holds, then $T$ is a weighted multi-Toeplitz operator. If $\varphi({\bf W},{\bf W}^*)=\sum_{{\bf s}\in \ZZ^k}
q_{\bf s}({\bf W},{\bf W}^*)$ is its Fourier representation, Theorem \ref{Fourier} shows that
$\varphi(r{\bf W},r{\bf W}^*)=\sum_{{\bf s}\in \ZZ^k}
q_{\bf s}(r{\bf W},r{\bf W}^*)$, $r\in [0,1)$, is convergent in the operator norm topology, $\sup_{r\in[0,1)}\|\varphi(r{\bf W},r{\bf W}^*)\|=\|T\|$
and $T=\text{\rm SOT-}\lim\limits_{r\to 1}F(r{\bf W}).
  $
Therefore, $F({\bf X}):=\varphi({\bf X},{\bf X})^*$, ${\bf X}\in {\bf D}_{{\bf f},rad}^{\bf m}(\cH)$, is a bounded free $k$-pluriharmonic function. The proof of the  last part of  the theorem is exactly the same as the one from Theorem 4.1 from \cite{Po-Toeplitz-poly-hyperball}. We omit it.
\end{proof}

We denote by ${\bf PH}_\cK^c({\bf D}_{{\bf f},rad}^{\bf m}) $ the set of all
  free $k$-pluriharmonic functions on the radial polydomain ${\bf D}_{{\bf f},rad}^{\bf m}$ with operator-valued coefficients in $B(\cK)$, which
 have continuous extensions   (in the operator norm topology) to
  ${\bf D_f^m}(\cH)$, for any Hilbert space $\cH$.
Here is the analogue of  the Dirichlet extension problem for polydomains.

\begin{theorem}\label{Dirichlet}  If $F:{\bf D}_{{\bf f},rad}^{\bf m}(\cH)\to B(\cK)\otimes_{min} B( \cH)$, then
 the following statements are equivalent.
\begin{enumerate}
\item[(i)] $F$ is a free $k$-pluriharmonic function on the radial polydomain ${\bf D}_{{\bf f},rad}^{\bf m}$
such that \ $F(r{\bf W})$ converges in the operator norm
topology, as $r\to 1$.

\item[(ii)]
There exists $T\in {\boldsymbol\cG}:=\text{\rm span}\left\{ C\otimes {\bf W}_{\boldsymbol\alpha}{\bf W}_{\boldsymbol\beta}^* :\   C\in B(\cK),   (\boldsymbol \alpha; \boldsymbol \beta)\in \boldsymbol\cJ\right\}^{\|\cdot\|}$ such
that $$F({\bf X})= \widetilde{\bf B}_{\bf X}[T], \qquad
 {\bf X}\in {\bf D}_{{\bf f},rad}^{\bf m}(\cH).
  $$
   \item[(iii)] $F$ is a free $k$-pluriharmonic function on the radial polydomain     ${\bf D}_{{\bf f},rad}^{\bf m}(\cH)$ which
 has a continuous extension  (in the operator norm topology) to the   polydomain
  ${\bf D}_{{\bf f}}^{\bf m}(\cH)$.

\end{enumerate}
In this case, $T=\lim\limits_{r\to 1}F(r{\bf W})$, where
the convergence is in the operator norm. Moreover, the map
$$
\Phi:{\bf PH}_\cK^c({\bf D}_{{\bf f}}^{\bf m} ) \to \boldsymbol\cG
 \quad \text{ defined
by } \quad \Phi(F):=T
$$ is a  completely   isometric isomorphism of
operator spaces.
\end{theorem}
\begin{proof} The equivalence  (i)$\leftrightarrow$(ii) is due to Theorem \ref{main}, Corollary \ref{norm}, Theorem
\ref{bounded},  and the properties of the noncommutative  Berezin transform. Since the proof is straightforward, we leave it to the reader. The implication (iii)$\rightarrow$(i) is clear while   the proof of the implication (ii)$\rightarrow$(iii) uses Theorem \ref{bounded}. Since the proof of the later implication  and the last part of the theorem is   exactly the same as the corresponding  proof of Theorem 4.3 from \cite{Po-Toeplitz-poly-hyperball}, we leave it to the reader.
\end{proof}

We denote by $\cS_\cK$  the norm-closed linear span of all the operators   of the form
$C\otimes {\bf W}_{\boldsymbol\alpha}{\bf W}_{\boldsymbol\beta}^*$, where $C\in B(\cK)$ and $\boldsymbol\alpha, \boldsymbol\beta\in {\bf F}_{\bf n}^+:=\FF_{n_1}^+\times \cdots \times \FF_{n_k}^+$.
The following result is an easy extension of  Theorem 2.4 from \cite{Po-Berezin1}. We omit the proof.

\begin{theorem} \label{funct-calc}
 If ${\bf X}\in {\bf D_f^m}(\cH)$,
then there is a unital completely contractive linear map $$\Psi_{\bf f,X}:\cS_\cK\to B(\cE)\bar\otimes_{min} B(\cH)$$ such that
$$
\Psi_{\bf f,X}(g)=\lim_{\delta\to 1}\widetilde {\bf B}_{\delta{\bf X}} [g], \qquad g\in \cS_\cK,
$$
where the limit is in the operator norm topology. If, in addition, ${\bf X}$ is a pure tuple in ${\bf D_f^m}(\cH)$, then
$$
\Psi_{\bf f,X}(g)=\widetilde{\bf B}_{\bf X}[g],\qquad  g\in \cS_\cK.
$$
\end{theorem}

In what follows, we show that the free $k$-pluriharmonic functions are characterized by a {\it mean value property}.

\begin{theorem} \label{mean-value}
If $F:{\bf D}_{{\bf f}, rad}^{\bf m}(\cH)\to
 B(\cK)\bar\otimes_{min} B( \cH)$ is a
free $k$-pluriharmonic function, then it has the mean value property, i.e
$$
F({\bf X})=\Psi_{{\bf f}, \frac{1}{r}{\bf X}}[F(r{\bf W})],\qquad {\bf X}\in r{\bf D_f^m}(\cH), r\in (0,1).
$$

Conversely, if a function
$$\varphi:[0,1)\to  \boldsymbol{\cG} :=\text{\rm span}\left\{ C\otimes {\bf W}_{\boldsymbol\alpha}{\bf W}_{\boldsymbol\beta}^* :\   C\in B(\cK),   (\boldsymbol \alpha, \boldsymbol \beta)\in \boldsymbol\cJ\right\}^{\|\cdot\|}
$$
satisfies the relation
$$
\varphi(r)=\widetilde{\bf B}_{\frac{r}{t}{\bf W}}[\varphi(t)],\qquad \text{for any }\  0\leq r<t<1,
$$
then the map $G:{\bf D}_{{\bf f}, rad}^{\bf m}(\cH)\to
 B(\cE)\bar\otimes_{min} B( \cH)$ defined by
 $$
 G({\bf X}):=\Psi_{{\bf f}, \frac{1}{r}{\bf X}}[\varphi(r)],\qquad {\bf X}\in r{\bf D_f^m}(\cH), r\in (0,1),
$$
  is a
free $k$-pluriharmonic function and  $G(r{\bf W})=\varphi(r)$ for any $r\in [0,1)$.
 \end{theorem}
\begin{proof}
Let  $F:{\bf D}_{f, rad}^m(\cH)\to
 B(\cK)\otimes_{min} B( \cH)$  be a
free pluriharmonic function with representation
\begin{equation*} F({\bf X})=\sum_{s_1\in \ZZ}\cdots \sum_{s_k\in \ZZ}
\sum_{(\boldsymbol \alpha, \boldsymbol \beta)\in \boldsymbol\cJ_{\bf s}} A_{(\boldsymbol \alpha, \boldsymbol \beta)}\otimes {\bf X}_{ \boldsymbol\alpha}{\bf X}_{\boldsymbol\beta}^*,\qquad {\bf X}\in {\bf D}_{{\bf f},rad}^{\bf m}(\cH),
\end{equation*}
 where the convergence is in the operator norm topology.  In particular, we have
  \begin{equation*} F(r{\bf W})=\sum_{s_1\in \ZZ}\cdots \sum_{s_k\in \ZZ}
\sum_{(\boldsymbol \alpha, \boldsymbol \beta)\in \boldsymbol\cJ_{\bf s}} A_{(\boldsymbol \alpha, \boldsymbol \beta)}\otimes r^{|\boldsymbol\alpha|+|\boldsymbol\beta|}{\bf W}_{ \boldsymbol\alpha}{\bf W}_{\boldsymbol\beta}^*,\qquad r\in [0,1),
\end{equation*}
 where the convergence is in the operator norm topology.
Let ${\bf X}\in r{\bf D}_{{\bf f}}^{\bf m}(\cH)$, $r\in (0,1)$,   and note that  $F(r{\bf W})\in \boldsymbol\cG$.
Consequently, using the continuity  of $F$ on  $r{\bf D}_{f}^m(\cH)$, we deduce that
\begin{equation*}
\begin{split}
\Psi_{{\bf f}, \frac{1}{r}{\bf X}}[F(r{\bf W})]&=\lim_{\delta\to 1}\widetilde{\bf B}_{\frac{\delta}{r}{\bf X}}[F(r{\bf W})]\\
&=\lim_{\delta\to 1} F(\delta {\bf X})=F({\bf X}),
\end{split}
\end{equation*}
where $r<\delta<1$ and the limits are in the operator norm topology.

Now, we
 prove the converse.  To this end, assume that  a function $\varphi:[0,1)\to  \boldsymbol\cG$ satisfies the relation
\begin{equation}
\label{varp}
\varphi(r)=\widetilde{\bf B}_{\frac{r}{t}{\bf W}}[\varphi(t)],\qquad \text{for any }\  0\leq r<t<1.
\end{equation}
Acording  to Theorem  \ref{Fourier} and Corollary \ref{norm}, $\varphi(r)$ is a  weighed right multi-Toeplitz operator with unique formal  Fourier  representation
$$\sum_{s_1\in \ZZ}\cdots \sum_{s_k\in \ZZ}
\sum_{(\boldsymbol \alpha, \boldsymbol \beta)\in \boldsymbol\cJ_{\bf s}}  r^{|\boldsymbol\alpha|+|\boldsymbol\beta|}A_{(\boldsymbol \alpha, \boldsymbol \beta)}(r)\otimes{\bf W}_{ \boldsymbol\alpha}{\bf W}_{\boldsymbol\beta}^*$$
for some operators $\{A_{(\boldsymbol \alpha, \boldsymbol \beta)}(r)\}_{(\boldsymbol\alpha,\boldsymbol\beta)\in \boldsymbol\cJ}$ in $B(\cK)$.
On the other hand, setting
$$
\varphi_\delta(r):=\sum_{s_1\in \ZZ}\cdots \sum_{s_k\in \ZZ}
\sum_{(\boldsymbol \alpha, \boldsymbol \beta)\in \boldsymbol\cJ_{\bf s}}  r^{|\boldsymbol\alpha|+|\boldsymbol\beta|}A_{(\boldsymbol \alpha, \boldsymbol \beta)}(r)\otimes \delta^{|\boldsymbol\alpha|+|\boldsymbol\beta|}{\bf W}_{ \boldsymbol\alpha}{\bf W}_{\boldsymbol\beta}^*, \qquad \delta\in [0,1),
$$
where  the convergence of the series is in the operator norm topology, Theorem \ref{Fourier}  shows that
\begin{equation}
\label{Vasu}
\varphi(r)=\text{\rm SOT-}\lim_{\delta\to 1}\varphi_\delta(r)\quad \text{ and }\quad  \sup_{\delta\in [0,1)}\|\varphi_\delta(r)\|=\|\varphi(r)\|
\end{equation}
for any $r\in [0,1)$.
 Taking into account that  the map $Y\to Y\otimes I$ is SOT-continuous on bounded sets, one can easily see that  the noncommutative Berezin transform $\widetilde{\bf B}_{\frac{r}{t}{\bf  W}}$ is SOT-continuous on bounded sets.
Due to   relation \eqref{varp}  and the results above,  we have
\begin{equation*}
\begin{split}
\varphi(r)&=\widetilde{\bf B}_{\frac{r}{t}{\bf W}}[\varphi(t)]
=\text{\rm SOT-}\lim_{\delta \to 1}
\widetilde{\bf B}_{\frac{r}{t}{\bf W}}[\varphi_\delta(t)]\\
&=\text{\rm SOT-}\lim_{\delta \to 1}\left(
 \sum_{s_1\in \ZZ}\cdots \sum_{s_k\in \ZZ}
\sum_{(\boldsymbol \alpha, \boldsymbol \beta)\in \boldsymbol\cJ_{\bf s}}  r^{|\boldsymbol\alpha|+|\boldsymbol\beta|}A_{(\boldsymbol \alpha, \boldsymbol \beta)}(t)\otimes \delta^{|\boldsymbol\alpha|+|\boldsymbol\beta|}{\bf W}_{ \boldsymbol\alpha}{\bf W}_{\boldsymbol\beta}^*
\right).
\end{split}
\end{equation*}
 Since, due \eqref{Vasu}, we also have
 $$
 \varphi(r)=\text{\rm SOT-}\lim_{\delta\to 1}\varphi_\delta(r)
 =
 \text{\rm SOT-}\lim_{\delta\to 1}\left( \sum_{s_1\in \ZZ}\cdots \sum_{s_k\in \ZZ}
\sum_{(\boldsymbol \alpha, \boldsymbol \beta)\in \boldsymbol\cJ_{\bf s}}  r^{|\boldsymbol\alpha|+|\boldsymbol\beta|}A_{(\boldsymbol \alpha, \boldsymbol \beta)}(r)\otimes \delta^{|\boldsymbol\alpha|+|\boldsymbol\beta|}{\bf W}_{ \boldsymbol\alpha}{\bf W}_{\boldsymbol\beta}^*\right),
 $$
  one can easily  see that $A_{(\boldsymbol \alpha, \boldsymbol \beta)}(r)=A_{(\boldsymbol \alpha, \boldsymbol \beta)}(t)$ for any $(\boldsymbol \alpha, \boldsymbol \beta)\in \boldsymbol \cJ$ and any $0\leq r<t<1$. Therefore, using again relation \eqref{Vasu}, we deduce that
$$
\varphi(r)=\sum_{s_1\in \ZZ}\cdots \sum_{s_k\in \ZZ}
\sum_{(\boldsymbol \alpha, \boldsymbol \beta)\in \boldsymbol\cJ_{\bf s}}  r^{|\boldsymbol\alpha|+|\boldsymbol\beta|}A_{(\boldsymbol \alpha, \boldsymbol \beta)}\otimes  {\bf W}_{ \boldsymbol\alpha}{\bf W}_{\boldsymbol\beta}^*, \qquad r\in [0,1),
$$
for some operators $\{A_{(\boldsymbol \alpha, \boldsymbol \beta)}\}_{(\boldsymbol\alpha,\boldsymbol\beta)\in \boldsymbol\cJ}$, where the convergence is in the operator norm topology.
Consequently,
$$
G({\bf X}):= \sum_{s_1\in \ZZ}\cdots \sum_{s_k\in \ZZ}
\sum_{(\boldsymbol \alpha, \boldsymbol \beta)\in \boldsymbol\cJ_{\bf s}}   A_{(\boldsymbol \alpha, \boldsymbol \beta)}\otimes  {\bf X}_{ \boldsymbol\alpha}{\bf X}_{\boldsymbol\beta}^*, \qquad {\bf X}\in {\bf D}_{{\bf f},rad}^{\bf m}(\cH),
$$
is a free $k$-pluriharmonic function
and $G(r{\bf W})=\varphi(r)$ for any $r\in [0,1)$.  Note also that
$$
 G({\bf X})=\Psi_{{\bf f}, \frac{1}{r}{\bf X}}[\varphi(r)]
$$
for any ${\bf X}\in r{\bf D_f^m}(\cH)$ and  $r\in (0,1)$.
The proof is complete.
\end{proof}

 Let  $G$  be a free $k$-pluriharmonic function with representation as above.  We associate with $G$  the  weighted multi-Toeplitz kernels $\Gamma_{rG}:{\bf F}_{\bf n}^+\times {\bf F}_{\bf n}^+\to B(\cK)$, $r\in[0,1)$,  defined by
 \begin{equation*}
\Gamma_{rF}(\boldsymbol\omega, \boldsymbol\gamma):=
\begin{cases}
\tau_{(\boldsymbol\omega,\boldsymbol\gamma)}r^{|{\bf s}(\boldsymbol \omega, \boldsymbol\gamma)|} A_{{\bf s}(\boldsymbol \omega, \boldsymbol\gamma)}, & \text{ if  } (\boldsymbol \omega, \boldsymbol\gamma)\in \boldsymbol\cC,\\
0,&  \text{ if  } (\boldsymbol \omega, \boldsymbol\gamma)\in ({\bf F}_{\bf n }^+\times {\bf F}_{\bf n }^+)\backslash \boldsymbol\cC,
\end{cases}
\end{equation*}
where the weights $ \{\tau_{(\boldsymbol\omega,\boldsymbol\gamma)}\}_{(\boldsymbol \omega, \boldsymbol\gamma)\in \boldsymbol\cC}$ are defined in Section 1.

Now, we prove a Schur type result  characterizing  the  positive free  $k$-pluriharmonic functions.

\begin{theorem}
Let $G:{\bf D}_{{\bf f}, rad}^{\bf m}(\cH)\to
 B(\cK)\bar\otimes_{min} B( \cH)$ be a free $k$-pluriharmonic function. Then the following statements  are equivalent.
 \begin{enumerate}
 \item[(i)] $G({\bf X})\geq 0$ for any ${\bf X}\in {\bf D}_{{\bf f}, rad}^{\bf m}(\cH)$.
 \item[(ii)] $G(r{\bf W})\geq 0$ for any $r\in [0,1)$.
 \item[(iii)]  The weighted multi-Toeplitz kernel $\Gamma_{rG}$ is positive semidefinite for any $r\in [0,1)$.
 \end{enumerate}
\end{theorem}
\begin{proof} Assume that $G$ has the representation
$$
G({\bf X}):= \sum_{s_1\in \ZZ}\cdots \sum_{s_k\in \ZZ}
\sum_{(\boldsymbol \alpha, \boldsymbol \beta)\in \boldsymbol\cJ_{\bf s}}   A_{(\boldsymbol \alpha, \boldsymbol \beta)}\otimes  {\bf X}_{ \boldsymbol\alpha}{\bf X}_{\boldsymbol\beta}^*, \qquad {\bf X}\in {\bf D}_{{\bf f},rad}^{\bf m}(\cH).
$$
   Since $G(r{\bf W})$ is a weighted multi-Toeplitz operator, a careful computation shows that
   $$
\left<G(r{\bf W})\left(\sum_{|\boldsymbol\beta|\leq q} h_{\boldsymbol\beta}\otimes e_{\boldsymbol\beta}\right), \sum_{|\boldsymbol\gamma|\leq q} h_{\boldsymbol\gamma}\otimes e_{\boldsymbol\gamma}\right>
=
\sum_{|\boldsymbol\beta|, |\boldsymbol\gamma|\leq q}
  \left<\Gamma_{rG}(\boldsymbol\gamma, \boldsymbol\beta) h_{\boldsymbol\beta}, h_{\boldsymbol\gamma}\right>
$$
for any $h_{\boldsymbol\beta}\in \cK$, where $\boldsymbol\beta\in {\bf F}_{n}^+$ and  $|\boldsymbol\beta|\leq q$.
Consequently, $G(r{\bf W})\geq 0$ for any $r\in [0,1)$ if and only if $\Gamma_{rG}$ is a positive semidefinite kernel for any $r\in [0,1)$. This shows that (ii) is equivalent to (iii).

 On the other hand, for each $X\in {\bf D}^{\bf m}_{{\bf f},rad}(\cH)$, there is $r\in (0,1)$ such that $X\in r{\bf D^m_f}(\cH)$. Due to Theorem \ref{mean-value}, we have
 $$
 G({\bf X})=\Psi_{{\bf f}, \frac{1}{r}{\bf X}}[G(r{\bf W})].
 $$
 Since the noncommutative Berezin transform is a positive map,  so is the map $\Psi_{{\bf f}, \frac{1}{r}{\bf X}}$.
 Now, it is clear that if $G(r{\bf W})\geq 0$ for any $r\in [0,1)$, then
 $ G({\bf X})\geq 0$ for any $X\in {\bf D}^{\bf m}_{{\bf f},rad}(\cH)$.
 The converse is obviously  true, which completes the proof of the equivalence of (i) with (ii).
 \end{proof}

Now, we prove  an analogue of Weierstrass  theorem  for free  $k$-pluriharmonic functions on the noncommutative polydomain ${\bf D}_{{\bf f}, rad}^{\bf m}(\cH)$.

\begin{theorem} \label{Weier} Let $F_s:{\bf D}_{{\bf f}, rad}^{\bf m}(\cH)\to B(\cK)\bar\otimes_{min} B( \cH)$, $s\in \NN$,   be a sequence of free  $k$-pluriharmonic functions such that, for any  $r\in [0,1)$,
 the sequence $\{F_s(r{\bf W}\}_{s=1}^\infty$
 is convergent in the operator norm topology.
  Then there is a free  $k$-pluriharmonic  function   $F:{\bf D}_{{\bf f}, rad}^{\bf m}(\cH)\to B(\cK)\bar\otimes_{min} B( \cH) $ such that $F_s(r{\bf W})$ converges to $F(r{\bf W})$, as $s\to \infty$,   for
any  $r\in [0,1)$. In particular, $F_s$ converges  to $F$ uniformly on any domain $r{\bf D_f^m}(\cH)$, $r\in [0,1)$.
\end{theorem}
\begin{proof}
Let $F_s$ have the representation
$$
F_s(X)= \sum_{s_1\in \ZZ}\cdots \sum_{s_k\in \ZZ}
\sum_{(\boldsymbol \alpha, \boldsymbol \beta)\in \boldsymbol\cJ_{\bf s}}   A_{(\boldsymbol \alpha, \boldsymbol \beta)}(s)\otimes  {\bf X}_{ \boldsymbol\alpha}{\bf X}_{\boldsymbol\beta}^*, \qquad {\bf X}\in {\bf D}_{{\bf f},rad}^{\bf m}(\cH),
$$
for some operators $\{A_{(\boldsymbol \alpha, \boldsymbol \beta)}(s)\}_{(\boldsymbol\alpha,\boldsymbol\beta)\in \boldsymbol\cJ}$, where the convergence is in the operator norm topology.
 If $r\in [0,1)$, then  $F_s(r{\bf W})$ is  in the operator space
 $$ \boldsymbol{\cG} :=\text{\rm span}\left\{ C\otimes {\bf W}_{\boldsymbol\alpha}{\bf W}_{\boldsymbol\beta}^* :\   C\in B(\cK),   (\boldsymbol \alpha, \boldsymbol \beta)\in \boldsymbol\cJ\right\}^{\|\cdot\|}.
  $$
Let  $\varphi:[0,1)\to  \boldsymbol{\cG} $ be defined by
\begin{equation}\label{fi-lim}
\varphi(r):=\lim_{s\to \infty} F_s(r{\bf W})\in \boldsymbol\cG,\qquad r\in [0,1).
\end{equation}
Note that, for any  $0 \leq r<t<1$,  we have
\begin{equation*}
\begin{split}
\widetilde{\bf B}_{\frac{r}{t}{\bf W}}[\varphi(t)]&=\lim_{s\to \infty}\widetilde{\bf B}_{\frac{r}{t}{\bf W}}[F_s(t{\bf W})]=\lim_{s\to \infty} F_s(r{\bf W})=\varphi(r),
\end{split}
\end{equation*}
where the limits are in the operator norm topology. Applying  Theorem \ref{mean-value},  we deduce  that the map $F:{\bf D}_{{\bf f}, rad}^{\bf m}(\cH)\to
 B(\cE)\bar\otimes_{min} B( \cH)$ defined by
 $$
 F({\bf X}):=\Psi_{{\bf f}, \frac{1}{r}{\bf X}}[\varphi(r)],\qquad {\bf X}\in r{\bf D_f^m}(\cH), r\in (0,1),
$$
  is a
free  $k$-pluriharmonic function on  ${\bf D}_{{\bf f},rad}^{\bf m}(\cH)$
 and  $F(r{\bf W})=\varphi(r)$ for any $r\in [0,1)$. Using relation \eqref{fi-lim}, we deduce that  $F(r{\bf W})=\lim_{k\to\infty} F_k(r{\bf W})$, $r\in [0,1)$.
 Due to the noncommutative von Neumann inequality for polydomains \cite{Po-Berezin1}, we have
 $$
 \sup_{{\bf X}\in r{\bf D_f^m}(\cH)} \|F({\bf X})-F_s({\bf X})\|= \|F(r{\bf W})-F_s(r{\bf W})\|,
 $$
 which implies  that $F_s$ converges  to $F$ uniformly on any domain $r{\bf D_f^m}(\cH)$, $r\in [0,1)$.
The proof is complete.
\end{proof}

\begin{corollary}  Let $G_s:{\bf D}_{{\bf f}, rad}^{\bf m}(\cH)\to B(\cK)\bar\otimes_{min} B( \cH)$, $s\in \NN$,   be a sequence of free pluriharmonic functions such that
$\{G_s(0)\}$ is a convergent sequence  in the  operator norm topology and
$$G_1\leq G_2\leq\cdots.
$$
Then $G_s$ converges to a free  $k$-pluriharmonic function on ${\bf D}_{{\bf f}, rad}^{\bf m}(\cH)$.
\end{corollary}
\begin{proof}Without loss of generality, we may assume that $G_1\geq 0$, otherwise we take $F_s:=G_s-G_1$, $s\in \NN$. Due to Harnack type inequality for positive free  $k$-pluriharmonic functions on ${\bf D}_{{\bf f}, rad}^{\bf m}(\cH)$ (see \cite{Po-Bohr-domains}),  if $s\geq q$, then we have
$$
\|G_s({\bf X})-G_q({\bf X})\|\leq \|G_s(0)-G_q(0)\| \frac{1-r}{1+r}
$$
for any $X\in r{\bf D_f^m}(\cH)$. Consequently,  since $\{G_s(0)\}$ is a Cauchy  sequence in the operator norm, we conclude that $\{G_s\}$ is a uniformly Cauchy sequence on $r{\bf D_f^m}(\cH)$.  Hence, $\{G_s(r{\bf W})\}$ is  convergent in the operator norm topology. Employing Theorem \ref{Weier}, we find  a free  $k$-pluriharmonic  function   $G:{\bf D}_{{\bf f}, rad}^{\bf m}(\cH)\to B(\cK)\bar\otimes_{min} B( \cH) $ such that $G_s(r{\bf W})$ converges to $G(r{\bf W})$, as $s\to \infty$,   for
any  $r\in [0,1)$. In particular, $G_s$ converges  to $G$ uniformly on any domain $r{\bf D_f^m}(\cH)$, $r\in [0,1)$.
The proof is complete.
\end{proof}

 We denote by  $Har_\cK( {\bf D}_{{\bf f}, rad}^{\bf m})$ denote the set of all free  $k$-pluriharmonic functions  $F:{\bf D}_{{\bf f}, rad}^{\bf m}(\cH)\to B(\cK)\bar\otimes_{min} B( \cH)$.
For $F,G\in Har_\cK( {\bf D}_{{\bf f}, rad}^{\bf m})$ and
$0<r<1$, we define
$$
d_r(F,G):=\|F(r{W})-G(r{W})\|.
$$
As before, we assume that  $\cH$ is
a separable  infinite dimensional  Hilbert space. The noncommutative von Neumann inequality for   the polydomain ${\bf D_f^m}(\cH)$ implies
$$
d_r(F,G)=\sup_{{X}\in r{\bf D_f^m}(\cH)}
\|F({X})-G({X})\|.
$$
 Let $\{r_s\}_{s=1}^\infty$ be an increasing sequence  of positive numbers
  convergent  to $1$.
For any $F,G\in Har_\cK( {\bf D}_{{\bf f}, rad}^{\bf m})$, we define
$$
\rho (F,G):=\sum_{s=1}^\infty \left(\frac{1}{2}\right)^s
\frac{d_{r_s}(F,G)}{1+d_{r_s}(F,G)}.
$$
It is easy to see that
      $\rho$ is a metric
on $Har_\cK( {\bf D}_{{\bf f}, rad}^{\bf m})$.

As a consequence of Theorem \ref{Weier}, we obtain the following result.  Since the proof is similar to that of Theorem 3.10 from \cite{Po-Toeplitz}, we omit it.
\begin{theorem}\label{complete-metric}
$\left(Har_\cK( {\bf D}_{{\bf f}, rad}^{\bf m}), \rho\right)$  is a complete metric space.
 \end{theorem}

      \bigskip

\section{Brown-Halmos type equation associated with weighted multi-Toeplitz operators}

Brown and Halmos \cite{BH} proved that a necessary and sufficient condition that an operator on the Hardy space $H^2(\DD)$ be a Toeplitz operator is that
  $$
  S^*TS=T,
  $$
  where $S$ is the unilateral shift on $H^2(\DD)$. In \cite{LO}, Louhichi and Olofsson obtain a Brown-Halmos type characterization of Toeplitz operators with harmonic symbols on the weighted Bergman space $A_m(\DD)$,
 the Hilbert space of all analytic functions on the unit disc $\DD$ with
$$
\|f\|^2:=\frac{m-1}{\pi}\int_\DD|f(z)|^2 (1-|z|^2)^{m-2}dz<\infty.
$$
 Their result was recently extended by Eschmeier and Langend\" orfer  \cite{EL} to the analytic functional  Hilbert space $H_m(\BB)$ on the unit ball $\BB\subset \CC^n$ given by the reproducing kernel $\kappa_m(z,w):=\left(1-\left<z,w\right>\right)^{-m}$ for  $z,w\in \BB$, where  $m\geq 1$.
 Recently,  we obtained  (see \cite{Po-Toeplitz-poly-hyperball} )  a Brown-Halmos  type characterizations of  the weighted multi-Toeplitz operators associated with noncommutative poly-hyperballs, i.e. when $f_i=Z_{i,1}+\cdots +Z_{i,n_i}$ for any $i\in \{1,\ldots, k\}$.

In this section, we prove that the weighted multi-Toeplitz operators satisfy a Brown-Halmos type equation associated with the polydomain ${\bf D_f^m}$.

For each $i\in \{1,\ldots, k\}$ and $j\in \{1,\ldots, n_i\}$, we define the {\it weighted right creation operators}
$\Lambda_{i,j}:F^2(H_{n_i})\to F^2(H_{n_i})$ by setting $\Lambda_{i,j}:= R_{i ,j}G_{i,j}$,
  where $R_{i,1},\ldots, R_{i,n_i}$ are
 the right creation operators on the full Fock space $F^2(H_{n_i})$ and the operators  $G_{i,j}:F^2(H_{n_i})\to F^2(H_{n_i})$, $j\in \{1,\ldots, n_i\}$,   are defined  by
 $$
G_{i,j}e^i_\alpha:=\sqrt{\frac{b_{i,\alpha}^{(m_i)}}{b_{i, \alpha g_j^i}^{(m_i)}}}
e^i_\alpha,\qquad
 \alpha\in \FF_{n_i}^+.
$$
 In this case, we have
\begin{equation*}
\Lambda_{i,\beta} e^i_\gamma= \frac {\sqrt{b_{i,\gamma}^{(m_i)}}}{\sqrt{b_{i,
\gamma \tilde\beta}^{(m_i)}}} e^i_{ \gamma \tilde \beta} \quad \text{
and }\quad \Lambda_{i,\beta}^* e^i_\alpha =\begin{cases} \frac
{\sqrt{b_{i,\gamma}^{(m_i)}}}{\sqrt{b_{i,\alpha}^{(m_i)}}}e_\gamma& \text{ if
}
\alpha=\gamma \tilde \beta \\
0& \text{ otherwise }
\end{cases}
\end{equation*}
 for any $\alpha, \beta \in \FF_{n_i}^+$, where $\tilde \beta$ denotes
 the reverse of $\beta=g^i_{j_1}\cdots g^i_{j_p}$, i.e.,
 $\tilde \beta=g_{j_p}^i\cdots g^i_{j_1}$.
  We introduce  the operator ${\bf \Lambda}_{i,j}$ acting on
$F^2(H_{n_1})\otimes\cdots\otimes F^2(H_{n_k})$ and given by
$${\bf \Lambda}_{i,j}:=\underbrace{I\otimes\cdots\otimes I}_{\text{${i-1}$
times}}\otimes \Lambda_{i,j}\otimes \underbrace{I\otimes\cdots\otimes
I}_{\text{${k-i}$ times}}.
$$
We set   ${\bf \Lambda}_i:=({\bf \Lambda}_{i,1},\ldots,{\bf \Lambda}_{i,n_i})$  for each $i\in \{1,\ldots, k\}$.
The $k$-tuple   ${\bf \Lambda}:=({\bf \Lambda}_1,\ldots, {\bf \Lambda}_k)$
 plays the role of the {\it right universal model for the  noncommutative polydomain}
${\bf D}_{\tilde {\bf f}}^{\bf m}$.

For each $i\in \{1,\ldots, k\}$, set  $\tilde f_i:=\sum_{\alpha\in \FF_{n_i}^+, |\alpha|\geq 1}a_{i,\tilde\alpha}Z_{i,\alpha}$.   We recall from \cite{Po-domains} that
$$\sum_{\alpha\in \FF_{n_i}^+, |\alpha|\geq 1}a_{i,\tilde\alpha}\Lambda_{i,\alpha} \Lambda_{i,\alpha}^*\leq I\quad \text{ and } \quad \left(id-\Phi_{\tilde f_i,\Lambda_i}\right)^{m_i}(I)=P_\CC,
$$
where $\Phi_{\tilde f_i,\Lambda_i}(X)=\sum_{\alpha\in \FF_{n_i}^+, |\alpha|\geq 1}a_{i,\tilde\alpha}\Lambda_{i,\alpha} X \Lambda_{i,\alpha}^*.
$
For each $i\in \{1,\ldots, k\}$, we
define $\Gamma_i:=\left\{ \alpha\in \FF_{n_i}^+: \ a_{i,\alpha}\neq 0\right\}$ and $N_i:=\text{\rm card}\,\Gamma_i$. The direct sum of $N_i$ copies of $\cH$ is denoted by $\cH^{(N_i)}$.  We associate with ${\bf \Lambda}_i:=({\bf \Lambda}_{i,1},\ldots, {\bf \Lambda}_{i,n_i})$ the row contraction $C_{\tilde f_i,{\bf \Lambda}_i}:(\otimes_{s=1}^k F^2(H_{n_s}))^{(N_i)}\to \otimes_{s=1}^k F^2(H_{n_s})$ defined by
$$
C_{\tilde f_i,{\bf \Lambda}_i}:=\left[\sqrt{a_{i,\tilde\alpha}}{\bf \Lambda}_{i,\alpha} \cdots \right]_{\alpha\in \Gamma_i},
$$
where the entries $\sqrt{a_{i,\tilde\alpha}}{\bf \Lambda}_{i,\alpha}$ are arranged in the lexicographic order of $\Gamma_i\subset \FF_{n_i}^+$.
If $Y\in B(\cH,\cK)$, we denote by ${\bf diag}_{N} (Y):\cH^{(N)}\to \cK^{(N)}$  the direct sum of $N$ copies of $Y$.

\begin{proposition} \label{WWW} For each  $i\in \{1,\ldots, k\}$, the operator $C_{\tilde f_i,{\bf \Lambda}_i}$ satisfies the following properties.
\begin{enumerate}
\item[(i)]  The operator $C_{\tilde f_i,{\bf \Lambda}_i}$ has closed range and
$$\text{\rm range}\,C_{\tilde f_i,{\bf \Lambda}_i}=\left(\otimes_{s=1}^{i-1} F^2(H_{n_s})\right)\otimes(F^2(H_{n_i})\ominus \CC)\otimes \left(\otimes_{s=i+1}^{k} F^2(H_{n_s})\right).
$$
\item[(ii)]  The operator
$$C_{\tilde f_i,{\bf \Lambda}_i}^*C_{\tilde f_i,{\bf \Lambda}_i}: \text{\rm range}\,C_{\tilde f_i,{\bf \Lambda}_i}^*\to \text{\rm range}\,C_{\tilde f_i,{\bf \Lambda}_i}^*$$
is invertible
    and
 the operator
  $$
  C_{\tilde f_i,{\bf \Lambda}_i}(C_{\tilde f_i,{\bf \Lambda}_i}^*C_{\tilde f_i,{\bf \Lambda}_i})^{-1}C_{\tilde f_i,{\bf \Lambda}_i}^*: \otimes_{s=1}^k F^2(H_{n_s})\to \otimes_{s=1}^k F^2(H_{n_s})
  $$ is the orthogonal projection of
$\otimes_{s=1}^k F^2(H_{n_s})$ onto
$$ \left(\otimes_{s=1}^{i-1} F^2(H_{n_s})\right)\otimes(F^2(H_{n_i})\ominus \CC)\otimes \left(\otimes_{s=i+1}^{k} F^2(H_{n_s})\right).
$$
\item[(iii)]  The following identity holds:
$$C_{\tilde f_i,{\bf \Lambda}_i}(C_{\tilde f_i,{\bf \Lambda}_i}^*C_{\tilde f_i,{\bf \Lambda}_i})^{-1}=C_{\tilde f_i,{\bf \Lambda}_i} {\bf diag}_{N_i}\left( \sum_{j=0}^{m_i-1} (-1)^j \left(\begin{matrix}  m_i\\j+1
\end{matrix}\right) \Phi_{\tilde f_i,{\bf \Lambda}_i}(I)\right)|_{\text{\rm range}\,C_{\tilde f_i,{\bf \Lambda}_i}^* }.
$$
\end{enumerate}
 \end{proposition}
 \begin{proof}
 Due to the definition of  ${\bf \Lambda}_i:=({\bf \Lambda}_{i,1},\ldots,{\bf \Lambda}_{i,n_i})$, it is easy to see that
 $$
 \text{\rm range}\ {\bf \Lambda}_i=\left(\otimes_{s=1}^{i-1} F^2(H_{n_s})\right)\otimes(F^2(H_{n_i})\ominus \CC)\otimes \left(\otimes_{s=i+1}^{k} F^2(H_{n_s})\right).
$$
Note that
$\text{\rm range}\,C_{\tilde f_i,{\bf \Lambda}_i}\subset \text{\rm range}\ {\bf \Lambda}_i$ and, for any $\varphi_1,\ldots, \varphi_{n_i}\in F^2(H_{n_i})$, we have
$$
C_{\tilde f_i,{\bf \Lambda}_i} \left[\begin{matrix} \varphi_1\\ \vdots\\ \varphi_{n_i}\\0\\ \vdots
\end{matrix}\right]=\sum_{j=1}^{n_i} \sqrt{a_{i,g^i_j}}{\bf \Lambda}_{i,j}\varphi_j=\text{\rm range}\ {\bf \Lambda}_i.
$$
 The later equality is due to the fact that $a_{i,g^i_j}>0$ for any $j\in \{1,\ldots, n_i\}$. Therefore, we have
 $\text{\rm range}\,C_{\tilde f_i,{\bf \Lambda}_i}= \text{\rm range}\ {\bf \Lambda}_i$, which proves item (i).

 To prove part (ii), note that since $C_{\tilde f_i,{\bf \Lambda}_i}$  has closed range, so does $C_{\tilde f_i,{\bf \Lambda}_i}^*$. Consequently, the operator
$$C_{\tilde f_i,{\bf \Lambda}_i}^*C_{\tilde f_i,{\bf \Lambda}_i}: \text{\rm range}\,C_{\tilde f_i,{\bf \Lambda}_i}^*\to \text{\rm range}\,C_{\tilde f_i,{\bf \Lambda}_i}^*$$
is invertible. On the other hand, if $y\in \text{\rm range}\,C_{\tilde f_i,{\bf \Lambda}_i}$,  then $y=C_{\tilde f_i,{\bf \Lambda}_i}x$ for some $x\in\text{\rm range}\,C_{\tilde f_i,{\bf \Lambda}_i}^*$ and
$$
C_{\tilde f_i,{\bf \Lambda}_i}(C_{\tilde f_i,{\bf \Lambda}_i}^*C_{\tilde f_i,{\bf \Lambda}_i})^{-1}C_{\tilde f_i,{\bf \Lambda}_i}^* y=C_{\tilde f_i,{\bf \Lambda}_i}(C_{\tilde f_i,{\bf \Lambda}_i}^*C_{\tilde f_i,{\bf \Lambda}_i})^{-1}C_{\tilde f_i,{\bf \Lambda}_i}^*C_{\tilde f_i,{\bf \Lambda}_i}x=C_{\tilde f_i,{\bf \Lambda}_i}x=y.
$$
If $y\in (\text{\rm range}\,C_{\tilde f_i,{\bf \Lambda}_i})^\perp  =\ker C_{\tilde f_i,{\bf \Lambda}_i}^*$, then
$C_{\tilde f_i,{\bf \Lambda}_i}(C_{\tilde f_i,{\bf \Lambda}_i}^*C_{\tilde f_i,{\bf \Lambda}_i})^{-1}C_{\tilde f_i,{\bf \Lambda}_i}^* y=0$. This completes the proof of item (ii).

  To prove item (iii), we recall   that
$$
 (id-\Phi_{\tilde f_i{\bf \Lambda}_i})^{m_i}(I)=\underbrace{I\otimes\cdots\otimes I}_{\text{${i-1}$
times}}\otimes {\bf P}_\CC\otimes \underbrace{I\otimes\cdots\otimes
I}_{\text{${k-i}$ times}},\qquad i\in \{1,\ldots, k\},
 $$
  where ${\bf P}_\CC$ is the
 orthogonal projection of  $  F^2(H_{n_i})$ onto $\CC 1\subset F^2(H_{n_i})$.  Consequently, using item (ii),   we deduce that

 \begin{equation*}
 \begin{split}
 C_{\tilde f_i,{\bf \Lambda}_i}(C_{\tilde f_i,{\bf \Lambda}_i}^*C_{\tilde f_i,{\bf \Lambda}_i})^{-1}C_{\tilde f_i,{\bf \Lambda}_i}^*  &= I_{\otimes_{s=1}^k F^2(H_{n_s})}-(id-\Phi_{\tilde f_i,{\bf \Lambda}_i})^{m_i}(I)\\
 &=
 \sum_{j=0}^{m_i-1} (-1)^j \left(\begin{matrix}  m_i\\j+1
\end{matrix}\right)\Phi^{j+1}_{\tilde f_i,{\bf \Lambda}_i}(I)\\
&=C_{\tilde f_i,{\bf \Lambda}_i}{\bf diag}_{N_i}\left( \sum_{j=0}^{m_i-1} (-1)^j \left(\begin{matrix}  m_i\\j+1
\end{matrix}\right)\Phi^j_{\tilde f_i,{\bf \Lambda}_i}(I)\right)C_{\tilde f_i,{\bf \Lambda}_i}^*.
\end{split}
 \end{equation*}
 Hence, part (iii) follows. The proof is complete.
\end{proof}

For each $i\in \{1,\ldots, k\}$, the operator $C_{\tilde f_i,{\bf \Lambda}_i}': {\text{\rm range}\,C_{\tilde f_i,{\bf \Lambda}_i}^* }\to \otimes_{s=1}^k F^2(H_{n_s})$
defined by
$$C_{\tilde f_i,{\bf \Lambda}_i}':=C_{\tilde f_i,{\bf \Lambda}_i}(C_{\tilde f_i,{\bf \Lambda}_i}^*C_{\tilde f_i,{\bf \Lambda}_i})^{-1}
$$ is called the {\it Cauchy dual }of $C_{\tilde f_i,{\bf \Lambda}_i}$.

\begin{definition}
An operator $T\in B(\otimes_{s=1}^k F^2(H_{n_s}))$ is said to have the Brown-Halmos property if
$$
C_{\tilde f_i,{\bf \Lambda}_i}^{\prime*}TC_{\tilde f_i,{\bf \Lambda}_i}'= {\bf P}_{{\text{\rm range}\,C_{\tilde f_i,{\bf \Lambda}_i}^* }}
{\bf diag}_{N_i}\left( \sum_{j=0}^{m_i-1} (-1)^j \left(\begin{matrix}  m_i\\j+1
\end{matrix}\right)\Phi^j_{\tilde f_i,{\bf \Lambda}_i}(T)\right)|_{{\text{\rm range}\,C_{\tilde f_i,{\bf \Lambda}_i}^* }}
$$
for any $i\in \{1,\ldots, k\}$.
\end{definition}

  If $\cK$ is a separable Hilbert space, we say that an operator  $T\in B\left(\cK\bigotimes \otimes_{s=1}^k F^2(H_{n_s})\right)$ satisfies  the Brown-Halmos condition if
\begin{equation} \label{BH}
C_{\tilde f_i,I\otimes{\bf \Lambda}_i}^{\prime*}TC_{\tilde f_i,I\otimes {\bf \Lambda}_i}'= {\bf P}_{{\text{\rm range}\,C_{\tilde f_i,I\otimes{\bf \Lambda}_i}^* }}
{\bf diag}_{N_i}\left( \sum_{j=0}^{m_i-1} (-1)^j \left(\begin{matrix}  m_i\\j+1
\end{matrix}\right)\Phi^j_{\tilde f_i,I\otimes{\bf \Lambda}_i}(T)\right)|_{{\text{\rm range}\,C_{\tilde f_i, I\otimes{\bf \Lambda}_i}^* }}
\end{equation}
for any $i\in \{1,\ldots, k\}$,
where   $I\otimes{\bf \Lambda}_i:=[I_\cK\otimes{\bf \Lambda}_{i,1}\cdots I_\cK\otimes {\bf \Lambda}_{i,n_i}]$.
We would like to find all the solutions of this operator  equation.

\begin{theorem} \label{th-BH} If $T\in B(\cK\otimes \otimes_{s=1}^k F^2(H_{n_s}))$ is a weighted multi-Toeplitz operator associated with the polydomain ${\bf D_f^m}$, then it satisfies the Brown-Halmos property.
\end{theorem}
\begin{proof}
First, we show that if ${\bf s}=(s_1,\ldots, s_k)\in \ZZ^k$, then
the operator
$$q_{\bf s}({\bf W},{\bf W}^*):=\sum_{(\boldsymbol \alpha, \boldsymbol \beta)\in \boldsymbol\cJ_{\bf s}} A_{(\boldsymbol \alpha, \boldsymbol \beta)}\otimes {\bf W}_{ \boldsymbol\alpha}{\bf W}_{\boldsymbol\beta}^*
$$
satisfies the Brown-Halmos condition.
Fix $i\in \{1,\ldots, k\}$ and assume that   $s_i\geq 0$.
Then
\begin{equation}
\label{diag1}
q_{\bf s}({\bf W},{\bf W}^*)C_{\tilde f_i,I\otimes{\bf \Lambda}_i}=C_{\tilde f_i,I\otimes{\bf \Lambda}_i} {\bf diag}_{N_i}\left(q_{\bf s}({\bf W},{\bf W}^*)\right)
\end{equation}
  and
  \begin{equation}
\label{diag2}
{\bf diag}_{N_i}\left(q_{\bf s}({\bf W},{\bf W}^*)\right){\bf diag}_{N_i}\left({\bf \Psi}_{\tilde f_i,I\otimes {\bf \Lambda}_i}(I) \right)={\bf diag}_{N_i}\left({\bf \Psi}_{\tilde f_i,I\otimes  {\bf \Lambda}_i}(q_{\bf s}({\bf W},{\bf W}^*)) \right),
  \end{equation}
  where
  $$
  {\bf diag}_{N_i}\left({\bf \Psi}_{\tilde f_i,I\otimes {\bf \Lambda}_i}(X) \right):={\bf diag}_{N_i}\left( \sum_{j=0}^{m_i-1} (-1)^j \left(\begin{matrix}  m_i\\j+1
\end{matrix}\right)\Phi^j_{\tilde f_i,I\otimes{\bf \Lambda}_i}(X)\right).
  $$
   Now, using Proposition \ref{WWW}, item (iii), and relations \eqref{diag1}, \eqref{diag2}, we deduce that
\begin{equation*}
\begin{split}
&(C_{\tilde f_i,I\otimes{\bf \Lambda}_i}^*C_{\tilde f_i,I\otimes{\bf \Lambda}_i})^{-1}C_{\tilde f_i,I\otimes{\bf \Lambda}_i}^*
q_{\bf s}({\bf W},{\bf W}^*)C_{\tilde f_i,I\otimes{\bf \Lambda}_i}(C_{\tilde f_i,I\otimes{\bf \Lambda}_i}^*C_{\tilde f_i,I\otimes{\bf \Lambda}_i})^{-1}\\
&\qquad =(C_{\tilde f_i,I\otimes{\bf \Lambda}_i}^*C_{\tilde f_i,I\otimes{\bf \Lambda}_i})^{-1}C_{\tilde f_i,I\otimes{\bf \Lambda}_i}^*
q_{\bf s}({\bf W},{\bf W}^*) C_{\tilde f_i,I\otimes{\bf \Lambda}_i} {\bf diag}_{N_i}\left( \sum_{j=0}^{m_i-1} (-1)^j \left(\begin{matrix}  m_i\\j+1
\end{matrix}\right) \Phi_{\tilde f_i,I\otimes{\bf \Lambda}_i}^j(I)\right)|_{\text{\rm range}\,C_{\tilde f_i,I\otimes{\bf \Lambda}_i}^* }\\
&\qquad = (C_{\tilde f_i,I\otimes{\bf \Lambda}_i}^*C_{\tilde f_i,I\otimes{\bf \Lambda}_i})^{-1}C_{\tilde f_i,I\otimes{\bf \Lambda}_i}^*
C_{\tilde f_i,I\otimes{\bf \Lambda}_i} {\bf diag}_{N_i}\left(q_{\bf s}({\bf W},{\bf W}^*)\right){\bf diag}_{N_i}\left({\bf \Psi}_{\tilde f_i, I\otimes {\bf \Lambda}_i}(I) \right)|_{\text{\rm range}\,C_{\tilde f_i,I\otimes{\bf \Lambda}_i}^* }\\
& \qquad ={\bf P}_{\text{\rm range}\,C_{\tilde f_i,I\otimes{\bf \Lambda}_i}^* }{\bf diag}_{N_i}\left({\bf \Psi}_{\tilde f_i, I\otimes {\bf \Lambda}_i}(q_{\bf s}({\bf W},{\bf W}^*) \right)|_{\text{\rm range}\,C_{\tilde f_i,I\otimes{\bf \Lambda}_i}^*}.
\end{split}
\end{equation*}
If   $s_i< 0$, then
\begin{equation}\label{diag3}
C_{\tilde f_i, I\otimes{\bf \Lambda}_i}^*q_{\bf s}({\bf W},{\bf W}^*)= {\bf diag}_{N_i}(q_{\bf s}({\bf W},{\bf W}^*))C_{\tilde f_i, I\otimes{\bf \Lambda}_i} ^*
\end{equation}
and
\begin{equation}\label{diag4}
{\bf diag}_{N_i}\left({\bf \Psi}_{\tilde f_i, I\otimes{\bf \Lambda}_i}(I) \right){\bf diag}_{N_i}(q_{\bf s}({\bf W},{\bf W}^*))={\bf diag}_{N_i}\left({\bf \Psi}_{\tilde f_i, I\otimes{\bf \Lambda}_i}(q_{\bf s}({\bf W},{\bf W}^*)) \right).
\end{equation}
Using using Proposition \ref{WWW}, item (iii), and relations \eqref{diag3}, \eqref{diag4}, we deduce that
\begin{equation*}
\begin{split}
&(C_{\tilde f_i,I\otimes{\bf \Lambda}_i}^*C_{\tilde f_i,I\otimes{\bf \Lambda}_i})^{-1}C_{\tilde f_i,I\otimes{\bf \Lambda}_i}^*
q_{\bf s}({\bf W},{\bf W}^*)C_{\tilde f_i,I\otimes{\bf \Lambda}_i}(C_{\tilde f_i,I\otimes{\bf \Lambda}_i}^*C_{\tilde f_i,I\otimes{\bf \Lambda}_i})^{-1}\\
&
 \qquad={\bf P}_{\text{\rm range}\,C_{\tilde f_i,I\otimes{\bf \Lambda}_i}^* }{\bf diag}_{N_i}\left( \sum_{j=0}^{m_i-1} (-1)^j \left(\begin{matrix}  m_i\\j+1
\end{matrix}\right) \Phi_{\tilde f_i,I\otimes{\bf \Lambda}_i}(I)\right)C_{\tilde f_i,I\otimes{\bf \Lambda}_i}^*
q_{\bf s}({\bf W},{\bf W}^*)C_{\tilde f_i,I\otimes{\bf \Lambda}_i}(C_{\tilde f_i,I\otimes{\bf \Lambda}_i}^*C_{\tilde f_i,I\otimes{\bf \Lambda}_i})^{-1}\\
&\qquad={\bf P}_{\text{\rm range}\,C_{\tilde f_i,I\otimes{\bf \Lambda}_i}^* }{\bf diag}_{N_i}\left({\bf \Psi}_{\tilde f_i, I\otimes{\bf \Lambda}_i}(I) \right){\bf diag}_{N_i}(q_{\bf s}({\bf W},{\bf W}^*))C_{\tilde f_i,I\otimes{\bf \Lambda}_i}^*C_{\tilde f_i,I\otimes{\bf \Lambda}_i}(C_{\tilde f_i,I\otimes{\bf \Lambda}_i}^*
C_{\tilde f_i,I\otimes{\bf \Lambda}_i})^{-1}\\
&\qquad={\bf P}_{\text{\rm range}\,C_{\tilde f_i,I\otimes{\bf \Lambda}_i}^* }{\bf diag}_{N_i}\left({\bf \Psi}_{\tilde f_i, I\otimes {\bf \Lambda}_i}(q_{\bf s}({\bf W},{\bf W}^*)) \right)|_{\text{\rm range}\,C_{\tilde f_i,I\otimes{\bf \Lambda}_i}^*}.
 \end{split}
\end{equation*}
Due to the first part of the proof, for any  $r\in [0,1)$, $q_{\bf s}(r{\bf W},r{\bf W}^*)=r^{|{\bf s}|}q_{\bf s}({\bf W},{\bf W}^*)$ satisfies the Brown-Halmos condition. According to Theorem \ref{main}, there is a   bounded free  $k$-pluriharmonic  function $F$ on  the radial polydomain ${\bf D}_{{\bf f},rad}^{\bf m}$ with coefficients in $B(\cK)$ such that
    $$T=\text{\rm SOT-}\lim_{r\to 1} F(r{\bf W})$$
and
$F(r{\bf W})=\sum_{{\bf s}\in \ZZ^k} r^{|{\bf s}|}q_{\bf s}({\bf W},{\bf W}^*)$ is convergent in the operator norm topology. It is easy to see that  $F(r{\bf W})$ satisfies the Brown-Halmos condition  \eqref{BH}.
Using the SOT-convergence above, we can prove that, for each $j\in \{0,1,\ldots, m_i-1\}$ and any $\boldsymbol\alpha, \boldsymbol\beta\in {\bf F}_{\bf n}^+$,  $h,\ell\in \cK$,
\begin{equation}
\label{conv-sot}
\left< \Phi^j_{\tilde f_i,I\otimes{\bf \Lambda}_i} (F(r{\bf W})) h\otimes e_{\boldsymbol\alpha}, \ell\otimes e_{\boldsymbol\beta}\right>    \to      \left<\Phi^j_{\tilde f_i,I\otimes{\bf \Lambda}_i }(T) h\otimes e_{\boldsymbol\alpha}, \ell\otimes e_{\boldsymbol\beta}\right> , \text{ as } \ r\to 1.
\end{equation}
Since
$$\Phi_{\tilde f_i,I\otimes{\bf \Lambda}_i}(I)=\sum_{\alpha\in \FF_{n_i}^+, |\alpha|\geq 1}a_{i,\tilde\alpha}\Lambda_{i,\alpha} \Lambda_{i,\alpha}^*\leq I
$$
and, due to Theorem \ref{main},  $\sup_{r\in[0,1)}\|F(r{\bf W})\|=\|T\|$,  we deduce that
$$\sup_{r\in [0,1)}\|\Phi^j_{\tilde f_i,I\otimes{\bf \Lambda}_i} (F(r{\bf W}))\|<\infty.
$$
Using relation \eqref{conv-sot}, we deduce that  $\Phi^j_{\tilde f_i,I\otimes{\bf \Lambda}_i} (F(r{\bf W})\to \Phi^j_{\tilde f_i,I\otimes{\bf \Lambda}_i }(T)$, as $r\to 1$, in the weak operator topology.
Consequently, and using that
\begin{equation*}
C_{\tilde f_i,I\otimes{\bf \Lambda}_i}^{\prime*}F(r{\bf W})C_{\tilde f_i,I\otimes {\bf \Lambda}_i}'= {\bf P}_{{\text{\rm range}\,C_{\tilde f_i,I\otimes{\bf \Lambda}_i}^* }}
{\bf diag}_{N_i}\left( \sum_{j=0}^{m_i-1} (-1)^j \left(\begin{matrix}  m_i\\j+1
\end{matrix}\right)\Phi^j_{\tilde f_i,I\otimes{\bf \Lambda}_i}(F(r{\bf W}))\right)|_{{\text{\rm range}\,C_{\tilde f_i, I\otimes{\bf \Lambda}_i}^* }}
\end{equation*}
for any $i\in \{1,\ldots, k\}$, we deduce that $T$ satisfies the Brown-Halmos condition.
The proof is complete.
 \end{proof}

 In \cite{Po-Toeplitz-poly-hyperball}, we proved that the converse of Theorem \ref{th-BH} is true for poly-hyperballs.  While we believe that the converse is true  for all  the noncommutative polydomains ${\bf D_f^m}$, for now, it remains an open problem.

\bigskip

       %

      \end{document}